\documentclass[11pt]{amsart}
\usepackage{amsmath,amsthm,amscd,amssymb}
\usepackage{geometry}
\usepackage{amssymb}

\usepackage{pb-diagram}
\usepackage{graphicx}
\usepackage{paralist}
\usepackage{graphicx}

\newtheorem{theorem}{Theorem}[section]
\newtheorem{corollary}[theorem]{Corollary}
\newtheorem{lemma}[theorem]{Lemma}
\newtheorem{proposition}[theorem]{Proposition}

\newtheorem*{new exact}{Theorem~\ref{thm:new exactsequence}}

\theoremstyle{definition}
\newtheorem{definition}[theorem]{Definition}
\newtheorem{example}[theorem]{Example}

\newtheorem{remark}[theorem]{Remark}

\newcommand{\R}{\mathbb{R}}

\newcommand{\Z}{\mathbb{Z}}
\newcommand{\C}{\mathbb{C}}
\newcommand{\K}{\mathcal{K}}
\newcommand{\M}{\mathcal{M}}
\newcommand{\I}{\mathcal{I}}
\newcommand{\rk}{\operatorname{rank}}

\newcommand{\G}{\Gamma}

\title[Higher-order signature cocycles]{Higher-Order Signature Cocycles for Subgroups of Mapping Class Groups and Homology Cylinders}

\author{Tim D. Cochran}
\address{Department of Mathematics MS-136, P.O. Box 1892, Rice University, Houston, TX 77251-1892}
\email{cochran@rice.edu}
\author{Shelly Harvey}
\address{Department of Mathematics MS-136, P.O. Box 1892, Rice University, Houston, TX 77251-1892}
\email{shelly@rice.edu}

\author{Peter D. Horn}
\address{Department of Mathematics, Columbia University MC 4406, 2990 Broadway, New York, NY 10027}
\email{pdhorn@math.columbia.edu}

\begin{document}






\begin{abstract}
We define families of invariants for elements of the mapping class group of $\Sigma$, a compact orientable surface.  Fix any characteristic subgroup $H\lhd \pi_1(\Sigma)$ and restrict to $J(H)$, any subgroup of mapping classes that induce the identity on $\pi_1(\Sigma)/H$. To any unitary representation $\psi$ of $\pi_1(\Sigma)/H$ we associate a higher-order $\rho_\psi$-invariant and a signature 2-cocycle $\sigma_\psi$. These signature cocycles are shown to be generalizations of the Meyer cocycle. In particular each $\rho_\psi$ is a quasimorphism and each $\sigma_\psi$ is a bounded $2$-cocycle on $J(H)$. In one of the simplest non-trivial cases, by varying $\psi$, we exhibit infinite families of linearly independent quasimorphisms and signature cocycles. We show that the $\rho_\psi$ restrict to homomorphisms on certain interesting subgroups.  Many of these invariants extend naturally to the full mapping class group and some extend to the monoid of homology cylinders based on $\Sigma$.
\end{abstract}

\maketitle

\section{Introduction}

Suppose $\Sigma$ is a compact oriented surface and $\mathcal{M}=\mathcal{M}(\Sigma)$ is its \emph{mapping class group},  i.e. the group of isotopy classes of orientation preserving diffeomorphisms of $\Sigma$ that restrict to the identity on $\partial \Sigma$. This includes the (framed) pure braid groups as one example. The mapping class group is important for several reasons. First, the classifying space $B\mathcal{M}$ is essentially homotopy equivalent to the moduli space of Riemann surfaces of topological type $\Sigma$. Furthermore, homeomorphisms of surfaces are very important in low-dimensional topology, since manifolds are often understood by decomposing them into simpler pieces.  For example, any $3$-manifold can be expressed as the union of two handlebodies identified along their common boundary surface via a homeomorphism. Similarly, recent attempts at a systematic structure for the study of $4$-manifolds view such manifolds as singular surface bundles over surfaces, called Lefschetz fibrations (and broken Lefschetz fibrations). Monodromies associated  to these fibrations are homeomorphisms of surfaces. These decompositions reduce the study of these complicated manifolds to the study of surface homeomorphisms. Our broad goal is to to describe and investigate many families of invariants for important subgroups of the mapping class groups using $3$- and $4$-dimensional manifolds. Many of our results also apply to subgroups of the monoid of homology cylinders, a recent generalization of $\mathcal{M}$.

Our invariants are generalizations of the classical Meyer signature cocycle \cite{May}, which we now briefly review. The Meyer signature cocycle has been defined only in the cases that the number of components of $\partial\Sigma$ is $0$ or $1$.  Recall that there is an exact sequence
\begin{eqnarray}\label{eq:standardexact}
1\to \mathcal{I}\overset{i}\longrightarrow \mathcal{M}\overset{r_M}\longrightarrow \textrm{Sp}(2g,\Z)\cong \mathbb{I}\mathrm{som}\left(H_1(\Sigma;\Z)\right)\to 1
\end{eqnarray}
where $r_M(f)$ is the induced action of $f$ on a fixed symplectic basis of $H_1(\Sigma)$,  $\mathbb{I}\mathrm{som}\left(H_1(\Sigma)\right)$ is the group of isometries of the intersection form on $H_1(\Sigma)$, and $\mathcal{I}$ is the \emph{Torelli group}. The latter is
the subgroup of $\M$ consisting of homeomorphisms that induce the identity on $H_1(\Sigma)$. Meyer defined a canonical $2$-cocycle
$$
\tau_M: \textrm{Sp}(2g,\Z)\times \textrm{Sp}(2g,\Z)\to \Z
$$
that induces a $2$-cocycle on $\mathcal{M}$ which we call the \emph{Meyer signature cocycle}
\begin{eqnarray}\label{eq:meyer}
\sigma_{M}:\mathcal{M}\times \mathcal{M}\xrightarrow{(r_M,r_M)} \textrm{Sp}(2g,\Z)\times \textrm{Sp}(2g,\Z)\overset{\tau_M}\longrightarrow\Z.
\end{eqnarray}
The Meyer cocycle satisfies the following properties that we call the \emph{Meyer properties}:
\begin{enumerate}
\item [1.] $\sigma_M$ is a \emph{bounded} $2$-cocycle (i.e. its range is bounded);
\item [2.] $\sigma_M(f,g)$ is the signature of the total space of the $\Sigma$-bundle over the twice punctured disk whose monodromy around the punctures is $f$ and $g$ respectively;
\item [3.] $\sigma_M$ vanishes as a $2$-cocycle on $\mathcal{I}$;
\end{enumerate}
Moreover, if genus$(\Sigma)\leq 2$ there is a (unique) corresponding $1$-chain, called the \emph{Meyer function},
$$
\rho_M:\mathcal{M}\to \mathbb{Q},
$$
such that $\delta\rho_M=\sigma_M$ in group cohomology with $\mathbb{Q}$-coefficients ~\cite{May}\cite[Equation 5.3]{A87}, and satisfying the following additional properties:
\begin{enumerate}
\item [4.] $\rho_M$ is a class function (i.e. it is constant on conjugacy classes);
\item [5.] $\rho_M$ is a quasimorphism (defined below);
\item [6.] the restriction of $\rho_M$ to $\mathcal{I}$ is a homomorphism.
\end{enumerate}
(In general if $\sigma_M$ is restricted to the \emph{hyperelliptic mapping class group} then such a Meyer's function exists with the above properties since $[\sigma_M]$ is trivial in the second rational cohomology of the hyperelliptic mapping class group ~\cite{Endo2,Mori2, Mori3}.)

The mathematics associated to Meyer's signature cocycle is extraordinarily rich.  Meyer himself gave formulae for the signature of surface bundles over surfaces and subsequent authors have extended these formulae to Lefschetz fibrations of $4$-manifolds and other complex varieties ~\cite{Endo1,Kuno2}.  Morita showed that $\sigma_M$ is part of a cocycle that is essentially equivalent to a Casson's celebrated invariant for homology $3$-spheres ~\cite{Mor2}.   As another example, Gambaudo-Ghys \cite{GaGh} consider the case of the braid group  and use their results to study the global geometry of the Gordian metric space of knots and to produce quasimorphisms on the group of compactly supported area-preserving diffeomorphisms of an open two-dimensional disc \cite{GaGh2}, and more generally to study the dynamics of surfaces \cite{Ghys}.

Quasimorphisms have been shown, in recent years, to be quite useful. Recall that a \textbf{quasimorphism} on a group $J$ is a function $\rho:J\to \mathbb{R}$ whose deviation from being a homomorphism is universally bounded by a constant $D_\rho$, that is, for all $f,g\in J$ 
$$
|~\rho(fg)-\rho(f)-\rho(g)~|\leq D_\rho.
$$
Two such are considered equivalent if they differ by a bounded function. Quasimorphisms are related to bounded cohomology (defined in Section~\ref{sec:propsigcocycles}), bounded generation ~\cite{BeFu,BeFu2} and stable commutator length \cite{Ba91,Kot04}. For example, if $\widehat{Q}(J)$ denotes the vector space of quasimorphisms of $J$ then there is an exact sequence
$$
0\to H^1(J;\mathbb{R})\to \widehat{Q}(J)\overset{\delta}\longrightarrow H^2_b(J;\mathbb{R})\to H^2(J;\mathbb{R}).
$$
An excellent place to learn about these subjects is \cite{Cal2}.

We assume throughout that $\Sigma$ is a surface with at least one boundary component, on one of which we choose a basepoint, $*$. We often denote $\pi_1(\Sigma,*)$, by $F$, a free group, whose rank will be suppressed (but is of course determined by the genus and the number of boundary components). Suppose $H$ is a characteristic subgroup of $F$. Then we let $J=J(H)$ denote the subgroup of $\mathcal{M}$ consisting entirely of homeomorphisms that induce the identity on $\pi_1(\Sigma,*)/H$. (Warning: this definition is only accurate if $\partial \Sigma$ has $1$ boundary component. See Section~\ref{sec:Defrho} for the correct definition of $J(H)$ in the cases that $\Sigma$ has more than one boundary component). For example $J(F)=\mathcal{M}$, and $J([F,F])=\mathcal{I}$. Another important example is $H=F_k$, the $k^{th}$ term of the lower central series of $\pi_{1}(\Sigma)$, $k\geq2$. In this case $J(H)$ is $\mathcal{J}(k)$, the $k^{th}$ \emph{generalized Johnson subgroup}, which is the subgroup of homeomorphisms that induce the identity on $F/F_k$. Specifically, in our notation $\mathcal{J}(2)$ is the Torelli group and $\mathcal{J}(3)$ is called the \emph{Johnson subgroup} (normally denoted $\mathcal{K}$). The $k^{th}$ term of the lower central series of $\I$ is another important subgroup. Yet another important class of examples are the mod $L$ versions of these subgroups. In particular, if $L\in \Z_+$ and $H=\bigcup_{x\in F}[F,F]x^L$, then $J(H)$ is the \emph{level L subgroup} of $\M$, sometimes denoted $\mathrm{Mod}(L)$, which is the subgroup of homomorphisms that induce the identity on $H_1(\Sigma;\Z/L\Z)$ \cite{Put1,Put2}. Other examples involve mixtures of the lower central and \emph{derived} subgroups of $F$.

Now fix a unitary representations $\psi: F/H \rightarrow U(\mathcal{H})$ on a separable Hilbert space $\mathcal{H}$ (one possibility is just a $U(n)$-representation). In Section~\ref{sec:Defrho} we give natural examples of such representations for some of the most important examples. To $H$ and $\psi$ we associate a \textbf{higher-order $\rho$-invariant}
$$
\rho_\psi: J(H)\to \mathbb{R}.
$$
In Section~\ref{sec:Defsigs} we define the \textbf{higher-order signature $2$-cocycle}
$$
\sigma_\psi:J(H)\times J(H)\to G,
$$
where $G=\Z$ if dim($\mathcal{H})<\infty$ and $G=\R$ if dim($\mathcal{H})=\infty$. In brief, the higher-order $\rho$-invariants are defined as follows: Given $f\in J(H)$, form the mapping torus $M_f$, which has a torus as its boundary.  From this perform ``longitudinal Dehn-filling'' to arrive at the closed $3$-manifold $N_f$. The latter is obtained by attaching a solid torus to $M_f$ in such a way that $*\times S^1$ bounds an embedded disk in $N_f$. We show that, under the hypothesis on $f$, there is a canonical surjection
$$
\phi_f:\pi_1(N_f)\to F/H.
$$
Given the pair $(N_f,\phi_f)$ and a fixed auxiliary \emph{finite-dimensional} unitary representation $\psi$, we let $\rho_\psi(f)=\rho(N_f,\psi\circ\phi_f)$ where the latter is the real-valued $\rho$-invariant of Atiyah-Patodi-Singer ~\cite{APS2}. In the infinite-dimensional case, we restrict to representations of the form
$$
\psi: F/H \rightarrow \G\overset{\ell_r}{\rightarrow}U(\ell^{(2)}(\G)),
$$
for a countable discrete $\G$ where $\ell_r$ is the left-regular representation of $\G$ on the Hilbert space $\ell^{(2)}(\G)$. In this case we set $\rho_\psi(f)=\rho(N_f,\psi\circ\phi_f)$, the \textbf{Cheeger-Gromov von Neumann $\rho$-invariant}  associated to $(N_f,\psi\circ\phi_f)$ ~\cite{ChGr1} (this is also called the $\ell^{(2)}-\rho$-invariant associated to $\psi\circ\phi_f$). These have the advantage that they are canonically associated to $H$ and hence enjoy better properties.

We establish that each of the $\rho_\psi$ and  $\sigma_\psi$ possess all of the Meyer properties

\begin{theorem}\label{thm:Meyerprops} For any $H$ and $\psi$ as above,
\begin{enumerate}
\item [0.] With real coefficients, $\delta\rho_\psi=\sigma_\psi$ (Proposition~\ref{prop:cobrhoissigma});
\item [1.] $\rho_\psi$ is a class function on $J(H)$ (Lemma~\ref{lem:classfunction});
\item [2.] $\rho_\psi$ is a quasimorphism on $J(H)$ (Proposition~\ref{prop:quasihomo});
\item [3.] $\sigma_\psi$ is a bounded $2$-cocycle on $J(H)$ (Theorem~\ref{thm:sigbounded}, Corollary~\ref{cor:boundedcocycle});
\item [4.] If $\Sigma$ has one boundary component then $\sigma_\psi(f,g)$ is the difference between a twisted signature and the ordinary signature of the total space of the $\Sigma$-bundle over the twice punctured disk whose monodromy around the punctures is $f$ and $g$ respectively (Corollary~\ref{cor:V=W});
\item [5.] If $\psi$ is finite-dimensional then $[\sigma_\psi]\in \ker(H^2(J(H);\Z)\to H^2(J(H);\mathbb{R}))$ (Corollary~\ref{cor:sigcohomologyclass});
\item [6.] $\sigma_\psi$ vanishes identically as a $2$-cocycle on $C(H)\cap \I$ (Corollary~\ref{cor:sigmavanishes}); where $C(H)\lhd J(H)$ is the subgroup consisting of those classes that induce the identity map
$$
\text{id}=f_*:\frac{H}{[H,H]}\to \frac{H}{[H,H]}.
$$
(see Definition~\ref{def:C(H)} for the definition of $C(H)$ when $\partial\Sigma$ is disconnected).
\item [7.] the restriction of $\rho_\psi$ to $C(H)\cap \I$ is a homomorphism (Corollary~\ref{cor:rhoishomo}),
\end{enumerate}

\end{theorem}

Moreover, in analogy to the exact sequence \eqref{eq:standardexact}:
\newtheorem*{newexact}{Theorem~\ref{thm:new exactsequence}}
\begin{newexact} If $\Sigma$ has one boundary component then there is an exact sequence
\begin{eqnarray}\label{eq:nonst1}
1\to C(H)\overset{i}\longrightarrow J(H)\overset{r_\psi}\longrightarrow \mathbb{I}\mathrm{som}\left(H_1(\Sigma;\mathbb{Z}[F/H])\right)
\end{eqnarray}
and a $2$-cocycle $\tau_\psi$ on the the image of $r_\psi$ such that
$$
\sigma_\psi=r_\psi^*(\tau_\psi)- n\sigma_M;
$$
where $n=$~dim($\mathcal{H})$ ($n=1$ if dim($\mathcal{H})=\infty$)
and $\mathbb{I}\mathrm{som}\left(H_1(\Sigma;\mathbb{Z}[F/H])\right)$ is the group of  automorphisms of $H_1(\Sigma;\mathbb{Z}[F/H])$ (as a $\mathbb{Z}[F/H]$-module) that preserve the intersection form with $\mathbb{Z}[F/H]$-coefficients ~\cite{Mi4}. 
\end{newexact}

The higher-order $\rho$-invariants and signature $2$-cocycles give a vast supply of invariants for subgroups of the mapping class group. In fact they yield maps
$$
\rho:\mathrm{Rep}(F/H,U(n))\to \widehat{Q}(J(H)),
$$
and
$$
\sigma:\mathrm{Rep}(F/H,U(n))\to H^2_b(J(H);\mathbb{R}).
$$

In certain cases, there is an interesting interpretation of $\rho_\psi$ as a twisted signature defect of a Lefschetz fibration \cite{Fuller}(or more generally of singular $\Sigma$-bundles over the $2$-disk):

\newtheorem*{introLefschetz}{Proposition~\ref{prop:Lefschetz}} \begin{introLefschetz} Suppose that $D_1,\dots,D_n$ are positive Dehn twists along null-homologous circles in $\Sigma$. Then, for any unitary representation $\psi$ of $F/[F,F]\equiv H_1(\Sigma;\Z)$,
$$
\rho_\psi(D_n\circ\dots\circ D_1)= \sigma(Y,\psi)-\sigma(Y)
$$
where $Y$ is the Lefschetz fibration over the $2$-disk with generic fiber $\Sigma$ and with $n$ singular fibers whose monodromies are $D_1,\dots,D_n$.
\end{introLefschetz}

Calculation of these invariants is, in general, difficult, as can be seen in \cite{GaGh,KM2}. However we include, in Section~\ref{sec:examples}, calculations in one of the simplest non-classical cases. Set $H=[F,F]$, choose a symplectic basis for $H_1(\Sigma;\Z)$ and define
$$
\psi_\omega:F/H\cong H_1(\Sigma;\mathbb{Z})\cong \Z^{2g}\overset{\pi}{\longrightarrow}S^1\equiv U(1),
$$
where, for each $i=1,...2g$, $\pi(x_i)=\omega$. Then, for each such $\omega$, we have the higher-order $\rho$-invariant $\rho_\omega=\rho_{\psi_\omega}$ defined on any subgroup of the Torelli group, $\mathcal{I}=J([F,F])$. Specifically, let $\mathcal{J}(3)=\K_g\subset\mathcal{I}$ be the Johnson subgroup.

\newtheorem*{introinfgenQH}{Theorem~\ref{thm:infgenquasi}}
\begin{introinfgenQH} For $g\geq 2$, $\{\rho_\omega\}$ spans an infinitely generated subspace of $\widehat{Q}(\K_g)$.
\end{introinfgenQH}

Previous constructions of quasimorphisms have used pure group theory, Seiberg-Witten theory, and quantum cohomology. Our construction is of a quite different flavor.

In addition,

\newtheorem*{infgenbounded}{Theorem~\ref{thm:infgenbounded}}
\begin{infgenbounded} For $g\geq 2$, $\{\sigma_\omega=\delta(\rho_\omega)\}$ spans an infinitely generated subspace of $H^2_b(\K_g;\mathbb{R})$, the second bounded cohomology of $\K_g$.
\end{infgenbounded}

It was recently shown in ~\cite{BeFu} that almost every subgroup of the mapping class group has infinite dimensional $H^2_b(-;\mathbb{R})$.  By contrast all the bounded cohomology groups of any amenable group vanish.

The subgroups on which the higher-order $\rho$-invariants \emph{are} homomorphisms promise to be very interesting. In particular, if $H=F_k$, then the groups $\{C([F_k,F_k])\}$, homeomorphisms that induce the identity on $F_k/[F_k,F_k]$ (and $F/F_k$), constitute a new and interesting filtration of the Torelli group. It was not known until recently whether or not $C([F_2,F_2])$ was non-empty, but it is now known that its intersection with each Johnson subgroup is non-zero, so $C([F_2,F_2])$  is highly non-trivial ~\cite{ChFa}.

We indicate a possible method of calculation that relies on previous work in link theory. There are various ways to map a punctured disk into $\Sigma$ and corresponding to these are ways to map the pure braid group into the Torelli group of $\Sigma$ ~\cite{Le8}. Let $\Theta$ be such a map. Then with some restrictions (see Proposition~\ref{prop:subsurfacerestrict}) the higher-order $\rho$-invariants of $\Theta(\beta)$ can be calculated in terms of the higher-order $\rho$-invariants of the zero framed surgery on the link obtained as the closure of the braid $\beta$. Such $\rho$-invariants of links have been studied extensively by the authors and others, although only a few calculations have been made for closures of pure braids ~\cite{C,Ha2,Hor2}. The recent thesis of M. Bohn may provide some tools for calculations in the general case \cite{Bohn}.

In Section~\ref{sec:invariantsforcylinders}, we generalize our work to the monoid of homology cylinders based on $\Sigma$, denoted $\mathcal{C}$ and to the group, $\mathcal{H}$, of \emph{homology cobordism classes of homology cylinders}  (defined in Section~\ref{sec:invariantsforcylinders}). These enlargements of $\mathcal{M}$ have been widely considered recently \cite{GarL1,Mor5,Sak1}.  

\section{Definition of the Higher-order $\rho$-invariants}\label{sec:Defrho}

In this section we will define the \textbf{higher-order $\rho$-invariant}
$$
\rho_\psi: J(H)\to \mathbb{R},
$$
associated to $H$ and $\psi$. Of course this serves to define such a function on any subgroup of $J(H)$. Basic properties of these invariants will be addressed in later sections.

\subsection{The subgroups $J(H)\subset\mathcal{M}$}\label{subsec:J(H)}

Suppose that $\Sigma$ is a connected oriented, compact surface with $m+1$ boundary components where $m\geq 0$. Choose a basepoint, $*$, on one of the boundary components, and basepoints $z_1,\dots,z_m$, on the other boundary components. Also choose directed arcs, $\delta_i$, in $\Sigma$ from $*$ to $z_i$. Recall that we are given $H$, a characteristic subgroup of $\pi_{1}(\Sigma,*)$, and $\psi:\pi_{1}(\Sigma) /H \rightarrow U(\mathcal{H})$, a unitary representation on a separable Hilbert space $\mathcal{H}$.

\begin{definition}\label{def:J(H)} Let $J=J(H)$ be the normal subgroup of $\mathcal{M}$ of mapping classes $[f]$ that satisfy
\begin{itemize}
\item [1.] $f$ induces the identity map on $\pi_{1}(\Sigma)/H$;
\item [2.] The homotopy classes $[f(\delta_i)\overline{\delta}_i]$ lie in $H$ for $1\leq i\leq m$. (Here $\overline{\delta}_i$ is the arc $\overline{\delta}_i$ backwards.)
\end{itemize}
\end{definition}
\noindent If $m=0$ then the second condition is vacuous. It is easy to check that the definition of $J(H)$ is independent of the choices of $*$, $z_i$ and $\delta_i$. For example, if $H=[F,F]$ then $J(H)$ is the Torelli group. The presence of condition $[2.]$ may be unfamiliar to the reader since  much of the literature deals with the case of a surface with a single boundary component ($m=0$). However, this is the ``right'' definition, even for the Torelli group (i.e. agrees with the definition of the Torelli group in ~\cite[p.114]{DJ1}).

\subsection{The associated $3$-manifolds}\label{subsection:3manifolds}

To define the $\rho$-invariants we first associate (in a standard fashion) to any $f\in J(H)$ a closed oriented $3$-manifold, $N_f$, and a canonical epimorphism $\phi_f:\pi_1(N_f)\to \pi_1(\Sigma)/H$.

We begin by recalling some notation. For any $f\in \mathcal{M}$, we can form the mapping torus of $f$, $M_{f}=\Sigma \times [0,1]/(x,0)\sim (f(x),1)$, a compact \emph{oriented} 3-manifold (possibly with boundary). The formation of $M_f$ is shown schematically by the first two pictures on the left side of Figure~\ref{fig:Mf}. In the schematic representation the vertical interval represents $\Sigma$ and the horizontal ``interval'' represents $[0,1]$.
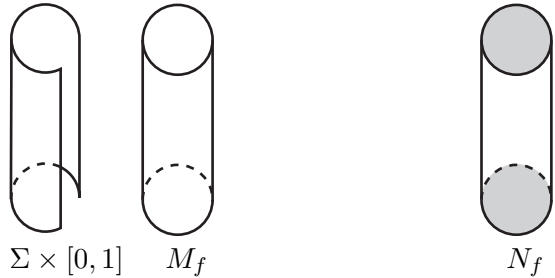
\begin{figure}[htbp]
\begin{picture}(105,110)(0,0)
\put(-60,12){\includegraphics{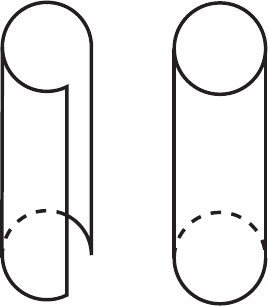}}
\put(118,12){\includegraphics{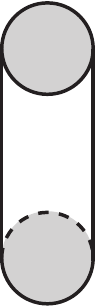}}
\put(-60,0){$\Sigma\times[0,1]$}
\put(-1,0){$M_f$}
\put(128,0){$N_f$}
\end{picture}
\caption{$M_f$ and $N_f$}
\label{fig:Mf}
\end{figure}
The oriented homeomorphism type of $M_f$ depends only on the \emph{conjugacy} class of $f$. More precisely, if $g$ and $f$ are conjugate then $M_g$ and $M_f$ are orientation-preserving homeomorphic relative to $(\Sigma \times \{0\})\cup \partial M_f$. It follows that for \emph{any} $f,g$, $M_{fg}\cong M_{gf}$. Each of the
boundary components of $M_f$ has a canonical identification with $S^{1}\times S^{1}$, where $S^1\times \{1\}$ is one of the components of $\partial \Sigma$, $t=\{*\}\times S^1$ is the circle $\{*\}\times [0,1]/\sim$ and $t_i=\{z_i\}\times S^1$ is the circle $\{z_i\}\times [0,1]/\sim$. Note that $M_{f^{-1}}=-M_f$ via an orientation-preserving homeomorphism fixing $(\Sigma \times \{0\})$ and inducing $(x,t)\to (x,-t)$ on each of the boundary tori. Figure~\ref{fig:Mf} is a representation of the case that $\Sigma$ has one circle boundary component (which appears as an $S^0$ in our schematic). Thus the top and bottom circles in the middle part of Figure~\ref{fig:Mf} represent the single boundary torus. If we attach solid tori to each of the boundary components of $M_f$ in such a way that $2$-disks are attached to the circles $\{*\}\times S^1$ and each $\{z_i\}\times S^1$, we denote this \emph{closed} manifold by $N_f$. It is shown schematically on the right-hand side of Figure~\ref{fig:Mf}, where the solid torus is shaded. This is the same as forming the quotient space $M_f\twoheadrightarrow  N_f$ wherein, for each $x\in \partial\Sigma$, $\{x\}\times S^1$ is identified to a single point. Given $f$, the 3-manifolds $M_f$ and $N_f$ are unique up to orientation-preserving homeomorphisms (relative $\partial M_{f}$ in the first case) that induce the identity on $\pi_{1}$.

Moreover, we have:
$$
\pi_{1}(M_f,*)\cong \langle\pi_{1}(\Sigma), t ~| ~t x t^{-1} = f_*(x),~ x\in \pi_{1}(\Sigma)\rangle,
$$
with respect to the canonical map $j_{\ast}: \pi_{1}(\Sigma\times \{0\}) \rightarrow \pi_{1}(M_f)$.  The subgroup $H$ is normal in $\pi_{1}(M_f)$ and
\begin{equation}\label{eq:coeffsyst}
\pi_{1}(M_f)/H\cong\langle\pi_{1}(\Sigma), t ~| ~t x t^{-1} = x, ~H\rangle\cong \Z\times \pi_{1}(\Sigma)/H
\end{equation}
since $f$ induces the identity modulo $H$. Since $N_f$ is obtained from $M_f$  by adding two cells along $\{t,t_1,\dots,t_m\}$, and then adding $3$-cells,
\begin{equation}\label{eq:pi1N}
\pi_{1}(N_f)\cong \langle\pi_{1}(\Sigma), t ~| ~t=1,~\delta_it_i\overline{\delta}_i=1, ~1\leq i\leq m, ~~x = f_*(x),~ x\in \pi_{1}(\Sigma) ~\rangle.
\end{equation}
The image of the rectangle $\delta_i\times [0,1]\hookrightarrow\Sigma \times [0,1]/\sim~= M_f$  shows that $t$ is based homotopic to $\delta_i t_if(\overline{\delta}_i)$. Then, using part $2$  of Definition~\ref{def:J(H)}, we have $f(\delta_i)\overline{\delta}_i=h_i$ so
$$
\delta_it_i\overline{\delta}_i=\delta_it_if(\overline{\delta}_i)f(\delta_i)\overline{\delta}_i\sim ~th_i.
$$
Thus, modulo $H$, the relations  $x = f_*(x)$ are trivial, and the relations $\delta_it_i\overline{\delta}_i$ are a consequence of the relation $t=1$. Hence,
\begin{equation}\label{eq:delta}
\pi_{1}(N_f)/H\cong\langle\pi_{1}(\Sigma)~|~ H\rangle\cong \pi_1(\Sigma)/H.
\end{equation}
Thus we see that there is a unique homomorphism
$$
\phi_f:\pi_1(N_f)\to \pi_{1}(\Sigma)/H
$$
such that the composition
$$
\pi_{1}(\Sigma)\overset{j_*}\to\pi_1(M_f)\to\pi_1(N_f)\overset{\phi_f}\to \pi_{1}(\Sigma)/H,
$$
is the canonical quotient map.

\subsection{The invariants}\label{subsec:invariants}

Now, given any fixed unitary representation $\psi:\pi_{1}(\Sigma)/H\to U(n)$, we get a canonical representation
$$
\psi_f:\pi_1(N_f)\to \pi_{1}(\Sigma)/H\overset{\psi}{\to}U(n).
$$
\noindent To any such pair $(N_f,\psi_f)$ Atiyah-Patodi-Singer associated a real-valued invariant $\rho(N_f,\psi_f)$, defined as a difference between the $\eta$ invariant of $N_f$ and a twisted $\eta$-invariant ~\cite{APS2}. These $\eta$ invariants are Riemannian spectral invariants associated to the signature operator, but the difference, $\rho(N_f,\psi_f)$, was shown to be an invariant of the oriented homeomorphism type of $(N_f,\psi_f)$. We call this the \textbf{higher-order $\rho$-invariant of f}  \textbf{corresponding to $\psi$}, denoted $\rho_\psi(f)$. 

Similarly, given any auxiliary $\phi: \pi_{1}(\Sigma)/H
\to \Gamma$ (for any countable discrete group $\Gamma
$) one can compose with the left-regular representation of $\Gamma$ on the Hilbert space $\ell^{(2)}(\Gamma)$, giving the representation
$$
\psi_f:\pi_1(N_f)\to \pi_{1}(\Sigma)/H\overset{\phi}{\to}\Gamma\to U(\ell^{(2)}(\Gamma)).
$$
To any such a pair $(N_f,\psi_f)$, Cheeger-Gromov associated a real number, $\rho(N_f,\psi_f)$, called the \textbf{von Neumann $\rho$-invariant} ~\cite{ChGr1}. Once again this was defined as the difference between the $\eta$ invariant of $N_f$ and the von Neumann $\eta$-invariant associated to the $\G$-cover of $N_f$. A summary of the basic properties of the $\rho$-invariants is given in Section~\ref{sec:sigprops}. As previously mentioned, the  von Neumann $\rho$-invariants have recently been extremely influential in the study of knots and links ~\cite{COT}.

In summary,

\begin{definition}\label{def:rho} The \textbf{higher-order $\rho$-invariant} of $f\in J(H)$  corresponding to $\psi$, denoted $\rho_\psi(f)$,  is $\rho(N_f,\psi_f)$  as above. Sometimes this will be abbreviated as $\rho(f)$  if $\psi$ is clear from the context.
\end{definition}

\begin{lemma}\label{lem:classfunction} For any $H$ and $\psi$, $\rho_\psi:J(H)\to \mathbb{R}$ is a class function on $J(H)$. Moreover, if $f\in J(H)$ and $g\in \M$, then $\rho_\psi(g^{-1}fg)=\rho_\psi(f)$.
\end{lemma}
\begin{proof} Since $J(H)$ is a normal subgroup of $\M$, $g^{-1}fg\in J(H)$. Then, as observed in Subsection~\ref{subsection:3manifolds}, $M_f\cong M_{g^{-1}fg}$, so  $\rho_\psi(g^{-1}fg)=\rho_\psi(f)$.
\end{proof}

\begin{example}\label{ex:easiestJ(H)} If $H$ is the commutator subgroup then $\pi_{1}(\Sigma)/H \cong H_1(\Sigma)\cong\mathbb{Z}^{\beta_1(\Sigma)}$ and $J(H)$ is the Torelli group. Given complex numbers of norm $1$, $\omega_i$, $1\leq i\leq \beta_1(\Sigma)$,  we can define $\psi_{\omega}:\mathbb{Z}^r\to U(1)\equiv S^1$  by sending $(0,\dots,1,\dots,0)$ to $\omega_i$. Therefore, varying the $\omega_i$ yields a function
$$
\rho:\left(S^1\times\dots\times S^1\right)\times J(H)\rightarrow \R,
$$
where here the $m$-torus should be viewed as the representation space $\mathrm{Rep}\left(\mathbb{Z}^{\beta_1(\Sigma)},U(1)\right)$. In addition the left-regular representation:
$$
\ell_r:\pi_{1}(\Sigma)/H=\mathbb{Z}^{\beta_1(\Sigma)}\hookrightarrow U\left(\ell^{(2)}(\mathbb{Z}^{\beta_1(\Sigma)})\right)
$$
gives a single function
$$
\rho^{(2)}:J(H)\to \R.
$$
It is known that this function is merely the integral over the $n$-torus of the function $\rho$ above ~\cite[Section 5]{COT}.
Furthermore suppose $\Sigma$ is the $2$-disk, $D^2$, with $m$ open subdisks deleted. Then, for any $f\in J(H)$, $M_f$ is homeomorphic to the exterior, $D^2\times S^1\setminus \hat{\beta}_f$ of the closure of an $m$-component pure braid $\beta_f$. The condition $f\in \mathcal{I}$ translates into the condition that the pairwise linking numbers of the components of $\hat{\beta}_f$ are zero (because the homology classes of the $[f(\delta_i)\overline{\delta}_i]$ encode the linking numbers of $\hat{\beta}_f$). Upon adding a solid torus to $M_f$ that caps off the boundary torus $\partial D^2\times S^1$, one arrives at the exterior, $S^3\setminus \hat{\beta}_f$. $N_f$ is obtained from this by adding $m$ additional solid tori (so called Dehn fillings) in such a way that the \emph{longitudes} of the components of $\hat{\beta}_f$ bound disks. The result is usually called the \emph{zero-framed surgery on the link $\hat{\beta}_f$ in $S^3$}, denoted here by $S(\hat{\beta}_f)$. The map $\psi_\omega$ is equivalent to assigning a complex number of norm $1$ to each (meridian) of the link $\hat{\beta_f}$. Therefore $\rho$ above yields
$$
\rho:\left(S^1\times\dots\times S^1\right)\times\mathcal{PB}_m^0\rightarrow \Z,
$$
where $\mathcal{PB}_m^0$ denotes the group of pure braids on $m$ strings with zero linking numbers. This function was (essentially) previously defined by Levine in ~\cite{Le10} for \emph{all links} (not just links that are the closures of pure braids) where it was shown that $\rho$ takes only integral values in this case (see also ~\cite{CiFl,Sm}). For example, if $K$ is a fixed knot then the function
$$
\rho_K:S^1\to \Z
$$
is precisely the \emph{Levine-Tristram signature function} of the knot $K$ and is given by the ordinary signature of the Hermitian matrix
$$
(1-\omega)V-(1-\overline{\omega})V^t
$$
where $V$ is a Seifert matrix for the knot. Therefore for knots and more generally for boundary links this function is straightforward to compute. Even here however the values are interesting as can be evidenced by ~\cite{GaGh} and recent papers addressing the values of this function for torus knots ~\cite{KM2,Boro1,Collins}. For general links, including the closures of pure braids, there is also a formula for this function in terms of bounding surfaces but there are almost no computations in the literature ~\cite{CiFl}. It is significant that the integral of this function is often much simpler than the function itself, as evidenced for torus knots.
\end{example}

\section{Definition of the Higher-order Signature Cocycles}\label{sec:Defsigs}

In this section we define the \textbf{higher-order signature $2$-cocycles}
$$
\sigma_\psi:~J(H)\times J(H)\to G
$$
where $G=\mathbb{Z}$ in the finite-dimensional case and $G=\mathbb{R}$ in the $\ell^{(2)}$ case. First we describe a $4$-manifold $V=V(f,g)$ and a closely related $4$-manifold $W=W(f,g)$, whose boundary is the disjoint union $N_{f} \sqcup N_{g} \sqcup -N_{fg}$. Then we show that the unitary representations extend over $\pi_1(V)$ and $\pi_1(W)$. We define $\sigma(f,g)$ to be a certain twisted signature defect of $W(f,g)$ corresponding to $\psi$. We later show that, in the important case that $\partial\Sigma$ is connected, the signature defects of $W(f,g)$ and $V(f,g)$ agree so that either may be used as the definition of $\sigma_\psi$.

Consider the $4$-manifold $M_f\times [0,1]$ as shown schematically on the left-hand side of Figure~\ref{fig:MfcrossI}. Let $V(f,g)$ denote the union of $M_f\times [0,1]$ and $M_g\times [0,1]$ identified along copies of $(\Sigma \times A )\times \{1\}$ in $M_f\times \{1\}$ and $M_g\times \{1\}$ where $A$ is a small interval about $\frac{1}{2}$ in $[0,1]/\sim$, so that
$$
\Sigma \times A \hookrightarrow \Sigma \times [0,1]\twoheadrightarrow \frac{\Sigma \times [0,1]}{\sim} \equiv M_f,
$$
(and we  have a similar copy $\Sigma\times A \hookrightarrow M_g$). This is shown on the right-hand side of Figure~\ref{fig:MfcrossI}.
\begin{figure}[htbp]
\begin{picture}(180,130)(0,0)
\put(-30,14){\includegraphics{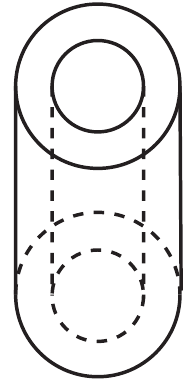}}
\put(140,14){\includegraphics{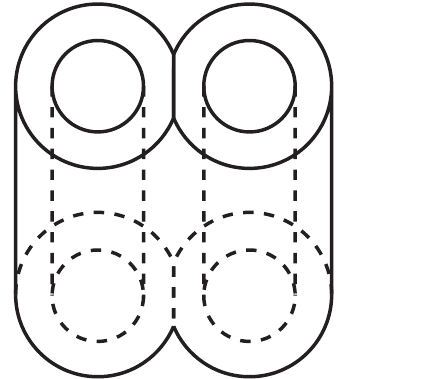}}
\put(-18,0){$M_f\times[0,1]$}
\put(176,0){$V(f,g)$}
\end{picture}
\caption{}
\label{fig:MfcrossI}
\end{figure}
Notice that $\partial V(f,g)$ contains copies of $M_f\cong M_f\times \{0\}$ and $M_g\cong M_g\times \{0\}$ (on the ``inside''), and also a copy of $M_{fg}$ (on the ``outside'').

There is an important alternative description of $V(f,g)$. Let $D$ be the closed oriented 2-disk with $2$ open subdisks deleted. This may be seen as a horizontal slice of $V(f,g)$ on the right-hand side of Figure~\ref{fig:MfcrossI}. Given $f,g\in J(H)$, we have a
unique homomorphism $\Phi: \pi_{1}(D)=\langle t_1,t_2\rangle \rightarrow J$ such that $\Phi(t_1)=f$ and $\Phi(t_2)=g$. This induces a unique (isomorphism class of) $\Sigma$-bundle over $D$. Since the bundle may be assumed to be a product over an arc $A$ that bisects $D$, it decomposes as the union of $M_f\times [0,1]$ and $M_g\times [0,1]$, intersecting along $A\times \Sigma$. Hence one sees that the total space of this bundle is identifiable with $V(f,g)$ defined above. In these terms the boundary of $V(f,g)$ is $M_f\sqcup M_{g} \sqcup - M_{fg} \cup
(\partial\Sigma\times D )$.

Now recall that
$$
N_f= M_f \bigcup_{\partial\Sigma\times S^1} \partial\Sigma\times D^2
$$
where $\partial\Sigma\times D^2$ is a disjoint union of $b$ solid tori where $b$ is the number of boundary components of $\Sigma$. Choose a small collar of $\partial\Sigma$ in $\Sigma$, $[0,\epsilon]\times \partial\Sigma\hookrightarrow \Sigma$. This induces
$$
A_f=[0,\epsilon]\times \partial\Sigma\times S^1\hookrightarrow M_f,
$$
a collar of $\partial M_f$. Now form the $4$-manifold
\begin{equation}\label{eq:plugs}
W(f,g)\equiv V(f,g)\bigcup_{A_f\times \{0\}}[0,\epsilon]\times \partial\Sigma\times D^2\bigcup_{A_g\times \{0\}}[0,\epsilon]\times \partial\Sigma\times D^2,
\end{equation}
as shown schematically in Figure~\ref{fig:W(f,g)}. Then $\partial(W(f,g))= N_{f} \sqcup N_{g} \sqcup -N_{fg}$ where the first two components are on the ``inside'', and the third is on the ``outside'' of the schematic representation. One can see a decomposition of $W(f,g)$ by bisecting the figure using a vertical plane, so that
$$
W(f,g)\cong \left(N(f)\times [0,1]\right)\cup_{\Sigma\times A}\left(N_g\times[0,1]\right).
$$
\begin{figure}[htbp]
\begin{picture}(210,140)(0,0)
\put(60,10){\includegraphics{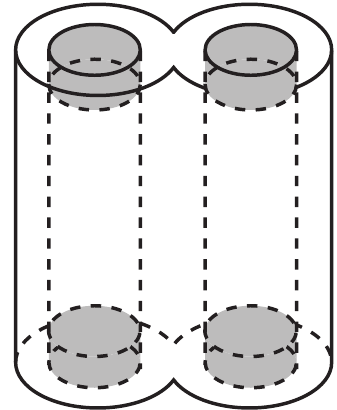}}
\end{picture}
\caption{$W(f,g)$}
\label{fig:W(f,g)}
\end{figure}

Using either point of view, the fundamental group of $V(f,g)$ has a presentation:
\[\langle\pi_{1}(\Sigma), t, s ~| ~t x t^{-1} = f_*(x), ~s x s^{-1} = g_*(x), ~x\in \pi_{1}(\Sigma)\rangle,\]
with respect to the canonical map $j_{\ast}: \pi_{1}(\Sigma) \rightarrow \pi_{1}(V(f,g))$.  The subgroup $H$ is normal in  $\pi_{1}(V(f,g))$ and $\pi_{1}(V(f,g))/H$ has a presentation
\[\langle\pi_{1}(\Sigma), t~, s ~|~ tx t^{-1} = x, ~ sx s^{-1} = x , ~H, ~x\in \pi_{1}(\Sigma)\rangle\]
since $f$ and $g$ induce the identity modulo $H$. But the addition of \eqref{eq:plugs} to $V(f,g)$ has the effect on $\pi_1$ of killing the $t$, $s$ as well as $t_i$, $s_i$ (as in Subsection~\ref{subsection:3manifolds}). If we kill these elements then we see that $j$ induces an isomorphism
\[j_{\ast}:\pi_{1}(\Sigma)/H \rightarrow \pi_{1}(V(f,g))/\langle H,t,t_{i}s,s_i\rangle\cong \pi_1(W(f,g)),\]
where we need the same analysis as was used for equation ~(\ref{eq:delta}).
Therefore, $\psi_f$ and $\psi_g$ extend uniquely to
$$
\tilde{\psi}: \pi_1(V(f,g))\twoheadrightarrow\pi_{1}(W(f,g)) \rightarrow U(\mathcal{H}).
$$

In summary, for any $f,g\in J(H)$, $\partial W(f,g)=N_f\sqcup N_g\sqcup -N_{fg}$ in such a way that for any unitary representation $\psi:\pi_1(\Sigma)/H\to U(\mathcal{H})$, there is a coefficient system, $\tilde{\psi}$, on $\pi_1(W(f,g))$ whose restriction to the boundary components is $\psi_f$, $\psi_g$ and $\psi_{fg}$ respectively. Similar statements hold for $V(f,g)\subset W(f,g)$ whose boundary is $M_f\sqcup M_{g} \sqcup - M_{fg} \cup
(\partial \Sigma\times D)$.

Recall that given $\widetilde{\psi}:\pi_1(W)\to U(n)$, where $W$ is a compact, connected orientable $4$-manifold $W$, one defines the twisted homology of $W$ as follows. Let $\widetilde{W}$ denote the universal cover of $W$ and consider a free left $\mathbb{Z}[\pi_1(W)]$ chain complex, $C_*(\widetilde{W})$, for $\widetilde{W}$. Note
$$
\widetilde{\psi}:\pi_1(W)\to U(n)\subset \mbox{Aut\,}\mathbb{C}^n
$$
endows $\mathbb{C}^n$ with the structure of a right $\mathbb{Z}[\pi_1(W)]$-module. Then set
$$
C_*(W;\widetilde{\psi})\equiv \mathbb{C}^n\otimes_{\widetilde{\psi}}C_*(\widetilde{W}).
$$
and
$$
H_*(W;\widetilde{\psi})\equiv H_*(C_*(W;\widetilde{\psi})).
$$
The usual intersection form on $H_2(W;\mathbb{C})$ generalizes to a hermitian form on $H_2(W;\widetilde{\psi})$ ~\cite[p.91]{L6}. The twisted signature $\sigma(W;\widetilde{\psi})$ is defined to be the ordinary signature of this hermitian form over $\mathbb{C}$. This signature takes values in $\mathbb{Z}$.

Similarly, given $\widetilde{\psi}:\pi_1(W)\to \Gamma\overset{\ell_r}\to U(\ell^{(2)}(\Gamma))$, the $\ell^{(2)}$-homology and the \textbf{von Neumann signature}, $\sigma^{2}_\Gamma(W;\widetilde{\psi})$, are defined (first defined by Atiyah in the case that $W$ is closed ~\cite{A76}, see \cite{LS}\cite[Section 5]{COT}). This signature takes values in $\mathbb{R}$. In Section~\ref{sec:sigprops} we will assemble, for the reader's convenience, the definition and basic properties of the von Neumann signature.

\begin{definition}\label{defcocycle} Given $H$ and a unitary representation $\psi:F/H\to U(\mathcal{H})$ as above we define, in case $\mathcal{H}$ has dimension $n$,
$$
\sigma_{\psi}:J(H)\times J(H)\to \Z,
$$
by
$$
\sigma_{\psi}(f,g)= \sigma(W(f,g);\widetilde{\psi})-n\sigma(W(f,g))
$$
and, in case $\mathcal{H}$ has dimension $\infty$, we define
$$
\sigma_{\psi}:J(H)\times J(H)\to \R,
$$
by
$$
\sigma_{\psi}(f,g)= \sigma^{(2)}_\G(W(f,g);\widetilde{\psi})-\sigma(W(f,g))
$$
where $W(f,g)$ is as defined in equation \eqref{eq:plugs} and $\sigma(W(f,g))$ is the signature of the ordinary intersection form on $H_2(W(f,g);\mathbb{C})$.
\end{definition}

\begin{remark}\label{rem:diffdefsignarturedefect} In the first case, it might be more natural use the definition
$$
\sigma_{\psi}(f,g)= \frac{\sigma(W;\widetilde{\psi})}{n}-\sigma(W)
$$
since then it is parallel to the $\ell^{(2)}$ case, being an ``average twisted signature'' minus an ordinary signature. But this leads to rational values of the signature, rather than integer values, so this explains our preference.
\end{remark}

\begin{proposition}\label{prop:easypropsigma}The following hold for $\sigma_{\psi}$.

\begin{enumerate}
\item $\sigma_\psi(f,g)=\sigma_\psi(g,f)$
\item $\sigma_\psi(f^{-1},g^{-1})=-\sigma_\psi(f,g)$
\item $\sigma_\psi(f,g)=0$ if $f=1$ or $g=1$ or $fg=1$.
\end{enumerate}
\end{proposition}

\begin{proof} By Definition~\ref{defcocycle}, $\sigma_\psi(f,g)$ is the twisted signature defect of the $4$-manifold $W(f,g)$ which has boundary $N_f\sqcup N_g\sqcup(-N_{fg})$ and $\sigma_\psi(g,f)$ is the twisted signature defect of the $4$-manifold $W(g,f)$ which has boundary $N_f\sqcup N_g\sqcup(-N_{gf})$. But, as previously observed, $N_{fg}=N_{gf}$. Thus $\partial W(f,g)=\partial W(g,f)$. Form the closed $4$-manifold $\overline{W}(f,g)=W(f,g)\cup -W(g,f)$. Since both the twisted signature and the ordinary signature are additive for manifolds glued along entire components of their boundaries, and since both signatures change sign upon changing orientation, the signature defect of $\overline{W}$ is
$$
\sigma_\psi(f,g)-\sigma_\psi(g,f).
$$
But by Atiyah's $L^{(2)}$-signature theorem ~\cite{A76} , the signature defect for a closed $4$-manifold is zero.  It follows that $\sigma_\psi(f,g)=\sigma_\psi(g,f)$.  The second property follows similarly by noting that
$$
\partial(-W(f,g))=-N_f\sqcup -N_g\sqcup\ N_{fg}=N_{f^{-1}}\sqcup N_{g^{-1}}\sqcup \ -N_{g^{-1}f^{-1}}=\partial W(f^{-1},g^{-1}),
$$
since $N_{g^{-1}f^{-1}}=N_{f^{-1}g^{-1}}$. The third property follows similarly upon noting that since
$\text{id}^{-1}=\text{id}$, $-N_{id}\cong N_{id}$, so
$$
\partial(-W(f,id))=-N_f\sqcup -N_{id}\sqcup N_{f}= N_{f}\sqcup N_{id}\sqcup -N_f\cong \partial(W(f,id).
$$
Thus $2\sigma_\psi(f,id)=0$. The other results follow similarly.
\end{proof}

We postpone  the proof that $\sigma_\psi$ satisfies the cocyle condition until Section~\ref{sec:propsigcocycles}, although it can be established using the ideas of the proof of Proposition~\ref{prop:easypropsigma}

We now observe that these signature cochains are intimately related to
the higher-order $\rho$-invariants.

\begin{proposition}\label{prop:rhoissigdefect} For each $\psi$
$$
\sigma_{\psi}(f, g)= -\rho_{\psi}(fg) +  \rho_{\psi}(f) +  \rho_{\psi}(g).
$$
where $\rho_\psi(f)$ is the higher-order $\rho$-invariant of f corresponding to $\psi$ as in Definition~\ref{def:rho}.
\end{proposition}

\begin{proof} Since
$$
\partial \left(W(f,g),\tilde{\psi}\right)=(N_f,\psi_f)\sqcup (N_g,\psi_g)\sqcup (-N_{fg},\psi_{fg}),
 $$
the proof follows immediately from our definition and the following results of Atiyah-Patodi-Singer (in the finite-dimensional case) and Ramachandran (in the $\ell^{(2)}$ case)(see also \cite{LS}).

\begin{theorem}\label{thm:APS}~\cite{APS2,Ra} Given a compact, smooth, orientable $4$-manifold $W$ and an extension $\widetilde{\psi} : \pi_1(W) \to U(n)$ of $\psi$ then
$$
\rho(\partial W,\psi)=\sigma(W,\widetilde{\psi})-n\sigma(W),
$$
where $\psi$ is the restriction of $\widetilde{\psi}$, $\sigma(W,\widetilde{\psi})$ is the signature of the twisted intersection form on $H_2(W;\widetilde{\psi})$ and $\sigma(W)$ is the signature of the ordinary intersection form on $H_2(W;\mathbb{C})$. Similarly given $\widetilde{\phi}: \pi_1(W)\to \Gamma$
$$
\rho(\partial W,\ell_r\circ\phi)=\sigma^{(2)}_\G(W,\ell_r\circ\widetilde{\phi})-\sigma(W).
$$
\end{theorem}
\end{proof}

One elementary consequence is:

\begin{corollary}\label{cor:easyrhoprops} $\rho_\psi(id)=0$.
\end{corollary}
\begin{proof} Merely apply Proposition~\ref{prop:rhoissigdefect} with $f=g=$id and then apply part $3$ of Proposition~\ref{prop:easypropsigma}.
\end{proof}

The following result is useful.

\begin{proposition}\label{prop:V=W} If $\partial\Sigma$ is connected, for any $H$, $\psi$, $f$ and $g$, the twisted and untwisted signatures of $V(f,g)$ and $W(f,g)$ are equal.
\end{proposition}

\begin{corollary}\label{cor:V=W} If $\partial\Sigma$ is connected then $\sigma_\psi(f,g)$ is the difference between the twisted signature and the ordinary signature of $V(f,g)$, which is the total space of the $\Sigma$-bundle over the twice punctured disk whose monodromy around the punctures is $f$ and $g$ respectively.
\end{corollary}

\begin{proof}[Proof of Proposition~\ref{prop:V=W}] Recall that $W=W(f,g)$ is obtained from $V=V(f,g)$ by adjoining a disjoint union of two thickened solid tori along a disjoint union of two thickened tori (one for $\partial M_f$ and one for $\partial M_g$). Since a solid torus is obtained from its boundary by adjoining a single $2$-handle and then a $3$-handle, the passage from $V$ to $W$ may be accomplished by adding two $2$-handles and then two $3$-handles. Let $\overline{W}$ denote the union of $V$ and these $2$-cells. We will show that $H_2(V)\cong H_2(\overline{W})$ with either twisted or untwisted coefficients. It suffices to show that
\begin{equation}\label{eq:MV0}
 H_1(S^1\sqcup S^1) \overset{i_*}{\longrightarrow} H_1(V)
\end{equation}
is injective where $S^1\sqcup S^1$ are the attaching circles $s$ and $t$ of the $2$-cells. Since $V\overset{\pi_*}\to D$ is a fibration, where $D$ is the $2$-disk with two open subdisks deleted, $D\to S^1\vee S^1$ is a deformation retraction and the map
$$
H_1(S^1\sqcup S^1) \overset{i_*}{\longrightarrow} H_1(V) \overset{\pi_*}\longrightarrow H_1(D)\to H_1(S^1\vee S^1)
$$
is the identity map. Note that the coefficient system $\psi$ is trivial on $S^1\sqcup S^1$, so
$$
H_1(S^1\sqcup S^1)\cong H_1(S^1\vee S^1),
$$
with twisted or untwisted coefficients. Thus $i_*$ is injective. The addition of $3$-handles will not change the signature since their attaching spheres are homology classes carried by the boundary of $\overline{W}$. This shows that the twisted and untwisted signatures of $W$ and $V$ agree.
\end{proof}

\section{Higher-order signature cocycles and group cohomology}\label{sec:propsigcocycles}

In this section we observe that each $\sigma_\psi$ is a bounded 2-cocycle in the group cohomology of $J(H)$, and, with $\R$-coefficients, $\sigma_\psi$ is the coboundary of $\rho_\psi$.

We review the definition of group cohomology with coefficents in a trivial module. If $G$ is a group and $A$ is an abelian group (viewed as a trivial $G$-module), set $G^p=G\times\cdots\times G$ and define the group of $A$-valued \textbf{p-cochains} to be
$$
C^p(G;A)=\{\rho:G^p\to A\}.
$$
Define $\delta:C^p(G;A)\to C^{p+1}(G;A)$ by
\begin{equation}\label{eq:coboundary}
\delta\rho(f_0,...,f_{p})=\rho(f_1,...,f_{p})+ \sum_{i=1}^p (-1)^i\rho(f_0,...,f_{i-1}f_i,...,f_p)+(-1)^{p+1}\rho(f_0,...,f_{p-1}).
\end{equation}
Then, $H^p(G;A)$, the \textbf{cohomology of $G$ with coefficients in $A$} is defined to be the homology of the complex $\{C^*(G;A),\delta\}$ ~\cite{BR}. A cochain with values in $A\subset\mathbb{R}$ is called a \textbf{bounded cochain} if its range is bounded as a subset of $\mathbb{R}$. The bounded cochains form a subcomplex $C^*_b(G;\mathbb{R})\subset C^*(G;\mathbb{R})$. The homology of this (co)-chain complex is called the \textbf{bounded cohomology of G}.

In the following $A$ is the trivial module trivial module where $A=\Z$ if $\psi$ is finite-dimensional and $A=\R$ if $\psi$ is infinite-dimensional.

\begin{proposition}\label{prop:cobrhoissigma}  Under the map $i^2_{\#}:C^2(J;A)\to C^2(J;\mathbb{R})$
$$
i^2_{\#}(\sigma_\psi)= \delta_\mathbb{R}(\rho_\psi).
$$
The subscript $\mathbb{R}$ is to emphasize that we are speaking of cohomology with real coefficients (since $\rho$ is real-valued). Thus, if $\psi$ is infinite-dimensional,  $\sigma_\psi$ is a 2-coboundary, while  if $\psi$ is finite-dimensional, $\sigma_\psi$ may not be a coboundary (with $\Z$ coefficients).
\end{proposition}
\begin{proof}  By equation ~(\ref{eq:coboundary}),
$$
(\delta_\mathbb{R}\rho_\psi)(f,g)= \rho_\psi(g)-\rho_\psi(fg)+\rho_\psi(f).
$$
The latter equals $\sigma_\psi(f,g)$ by Proposition~\ref{prop:rhoissigdefect}.
\end{proof}

\begin{corollary}\label{cor:cocycle} The function $\sigma_{\psi}=\sigma\hspace{-2pt}:\hspace{-2pt}J\times J\to A$ given by $(f,g)\to \sigma_\psi(f,g)$  is a $2$-cocycle of $J$ with values in $A$.
\end{corollary}
\begin{proof}  The map $i^3_{\#}:C^3(J;A)\to C^3(J;\mathbb{R})$ is injective. Thus it suffices to show that $i^3_{\#}\circ \delta_A(\sigma_\psi)=0$. By naturality and Proposition~\ref{prop:cobrhoissigma} ,
$$
i^3_{\#}\circ \delta_A(\sigma_\psi)=\delta_\mathbb{R}\circ i^2_{\#}(\sigma_\psi)= \delta_\mathbb{R}\delta_\mathbb{R}(\rho_\psi)=0.
$$
This corollary can also be proved directly from the definition of $\sigma_\psi$, using additivity properties of the signature.
\end{proof}

\begin{corollary}\label{cor:sigcohomologyclass} If $\psi$ is a finite-dimensional representation then the signature cocycle $\sigma_{\psi}$ represents an element in the kernel of
$$
H^2(J;\mathbb{Z})\to H^2(J;\mathbb{R}).
$$
\end{corollary}

\begin{theorem}\label{thm:sigbounded} For any n-dimensional representation $\psi$ ,
$$
|\sigma_\psi(f,g)|\leq 2n\beta_1(\Sigma).
$$
In the infinite-dimensional case,
$$
|\sigma_\psi(f,g)|\leq 2\beta_1(\Sigma).
$$
\end{theorem}

\begin{corollary}\label{cor:boundedcocycle} For any $\psi$, $\sigma_\psi$ is a bounded $2$-cocycle and hence represents an element in the kernel of
$$
H^2_b(J;\mathbb{R})\to H^2(J;\mathbb{R}).
$$
\end{corollary}
\begin{proof}[Proof of Corollary~\ref{cor:boundedcocycle}] By Corollary~\ref{cor:cocycle} and Theorem~\ref{thm:sigbounded}, $\sigma_\psi$ is a bounded $2$-cocycle. By Corollary~\ref{cor:sigcohomologyclass}, it vanishes in $H^2(J;\mathbb{R})$.
\end{proof}

\begin{proof}[Proof of Theorem~\ref{thm:sigbounded}] Recall the description of $W=W(f,g)$ of Figure~\ref{fig:W(f,g)}. By contracting along the thickenings, we see that, up to homotopy equivalence, $W\simeq N_f\cup_\Sigma N_g$. Thus we have the Mayer-Vietoris sequence below, which we consider with various coefficients.
\begin{equation}\label{eq:MV}
H_2(N_f)\oplus H_2(N_g)\overset{(i_*+j_*)}{\longrightarrow} H_2(W)\overset{\partial_*}{\longrightarrow} H_1(\Sigma) \overset{(i_*,j_*)}{\longrightarrow} H_1(N_f)\oplus H_1(N_g).
\end{equation}
Since the intersection form on $W$ with any coefficients is identically zero on $i_*(H_2(\partial W))$, it descends to a form on the quotient module
$$
H_2(W)/i_*(H_2(\partial W))
$$
and our various signatures are equal to the appropriate signature of this induced form. 

First, consider  the case that $\psi$ is an $n$-dimensional representation. By definition
$$
\sigma_{\psi}(f,g)= \sigma\left(W(f,g);\widetilde{\psi}\right)-n~\sigma(W(f,g)).
$$

By our remark above,
$$
\left|\sigma\left(W;\widetilde{\psi}\right)\right|\leq \rk_\mathbb{C}\left(H_2\left(W;\widetilde{\psi}\right)\Big/i_*\left(H_2\left(\partial W;\widetilde{\psi}\right)\right)\right).
$$
Considering ~\eqref{eq:MV} with $\C^n$-coefficients twisted by $\widetilde{\psi}$, we see that
$$
\rk_\mathbb{C}\left(H_2\left(W;\widetilde{\psi}\right)\Big/i_*\left(H_2\left(\partial W;\widetilde{\psi}\right)\right)\right)\leq \rk_\mathbb{C}(\mbox{image\ }\partial_*)\leq \rk_\mathbb{C} H_1\left(\Sigma;\widetilde{\psi}\right).
$$
Since $\Sigma$ has a cell decomposition with one zero cell and $\beta_1(\Sigma)$ one cells
$$
\rk_\mathbb{C} H_1\left(\Sigma;\widetilde{\psi}\right)\leq \rk_\mathbb{C} C_1\left(\Sigma;\widetilde{\psi}\right)= \rk_\mathbb{C} \left(\mathbb{C}^n\otimes_{\widetilde{\psi}}\left(\Z[\pi_1(W)]^{\beta_1(\Sigma)}\right)\right)=n\beta_1(\Sigma).
$$
Hence we have shown that
$$
|\sigma_\psi(f,g)|\leq 2n\beta_1(\Sigma).
$$
Taking $\psi$ to be a trivial $1$-dimensional representation, 
$$
|\sigma(W(f,g))|\leq \beta_1(\Sigma).
$$
This finishes the proof in the case of a finite-dimensional representation.

Now suppose $\psi$ is an infinite-dimensional representation.
Thus $\widetilde{\psi}:\pi_1(W)\to F/H\equiv\Gamma\overset{\ell_r}\to U(\ell^{(2)}(\Gamma))$ and by definition
$$
\sigma_{\psi}(f,g)= \sigma^{(2)}_\G\left(W;\widetilde{\psi}\right)-\sigma(W).
$$
and $\sigma^{(2)}_\G(W)$ is equal to the von Neumann signature of the induced form on
$$
H_2(W;\mathcal{U}\G)/i_*(H_2(\partial W;\mathcal{U}\G)).
$$
Since the von Neumann dimension is additive on short exact sequences (see this and other properties in ~\cite[Lemma 8.27, Assumption 6.2,Theorem 6.7]{Luc}),
$$
|\sigma^{(2)}_\G(W)|\leq \dim^{(2)}_\G\Big(H_2(W;\mathcal{U}\G)/i_*(H_2(\partial W;\mathcal{U}\G))\Big).
$$
Considering the sequence \eqref{eq:MV} with $\mathcal{U}\G$-coefficients, we see that
$$
\dim^{(2)}_\G\Big(H_2(W;\mathcal{U}\G)/i_*(H_2(\partial W;\mathcal{U}\G))\Big)\leq \dim^{(2)}_\G(\mbox{image}\ \partial_*)\leq \dim^{(2)}_\G H_1(\Sigma;\mathcal{U}\G).
$$
Furthermore
$$
\dim^{(2)}_\G H_1(\Sigma;\mathcal{U}\G)\leq \dim^{(2)}_\G C_1(\Sigma;\mathcal{U}\G)=\dim^{(2)}_\G (\mathcal{U}\G)^{\beta_1(\Sigma)}=\beta_1(\Sigma).
$$
Hence we have shown that
$$
|\sigma_\psi(f,g)|\leq 2\beta_1(\Sigma).
$$

\end{proof}

\begin{subsection}{Higher-order $\rho$-invariants as quasimorphisms}\label{subsesec:rhoinvts}\

We show that each of the higher-order $\rho$-invariants is a quasimorphism. In Section~\ref{sec:examples} we will show that even the very simplest family of such higher-order $\rho$-invariants spans an infinite-dimensional subspace of the the vector space, $\widehat{Q}(\mathcal{J}(3))$, of all quasimorphisms of $\mathcal{J}(3)$ (recall $\mathcal{J}(3)$ is the Johnson subgroup $\mathcal{K}$); and that 
the set of their coboundaries, $\{\delta(\rho_\omega)\}$ spans an infinitely generated subspace of $H^2_b(\mathcal{K};\mathbb{R})$, the second bounded cohomology of $\mathcal{K}$.

\begin{proposition} \label{prop:quasihomo} Each of the higher-order $\rho$-invariants, $\rho_\psi:J(H)\to \mathbb{R}$ is a
quasimorphism.
\end{proposition}

\begin{proof} Suppose $\rho=\rho_\psi:J(H)\to \mathbb{R}$ is a higher-order $\rho$-invariant. Then, by Proposition~\ref{prop:rhoissigdefect}, for each $f,g$
$$
|~\rho(fg)-\rho(f)-\rho(g)~|=|\sigma_\psi(f,g)~|.
$$
where $\sigma$ is the signature cocycle from Section~\ref{sec:Defsigs}. By Theorem~\ref{thm:sigbounded}, the latter is bounded independent of $f$ and $g$.
\end{proof}

\end{subsection}

\begin{subsection}{Subgroups on which the Higher-Order Signature Cocycles Vanish}\label{subsesec:rhohomo}\
\\

By examining the proof of Theorem~\ref{thm:sigbounded} we can draw more precise conclusions in certain cases.

\begin{definition}\label{def:C(H)} Let $C(H)\lhd J(H)$ denote the subgroup consisting of those classes $[f]$ such that
\begin{itemize}
\item [1.] $f$ induces the identity map on $H/[H,H]$; and
\item [2.] the homotopy classes $[f(\delta_i)\overline{\delta}_i]$ lie in $[H,H]$ for $1\leq i\leq m$ (compare Subsection~\ref{subsec:J(H)}).
\end{itemize}
\end{definition}

\begin{theorem}\label{thm:additivity}\
\begin{itemize}
\item [1.]If either of $f$ or $g$ lies in $C(H)$ then the twisted signature of $W(f,g)$ vanishes, i.e. in the finite case $\sigma(W(f,g);\widetilde{\psi})= 0$; and in the infinite case $\sigma_\G^{(2)}(W(f,g),\psi)=0$.
\item [2.] If either of $f$ or $g$ lies in $\mathcal{I}$ then the ordinary signature of $W(f,g)$ vanishes.
\end{itemize}
\end{theorem}

Before proving Theorem~\ref{thm:additivity}, we point out some of its interesting corollaries.

\begin{corollary}\label{cor:sigmavanishes} The signature defect $\sigma_\psi$ vanishes identically as a $2$-cocycle on $\mathcal{C}(H)\cap \I$.
\end{corollary}

\begin{proof}[Proof of Corollary~\ref{cor:sigmavanishes}] Recall that if dim($\mathcal{H})=n$, then
$$
\sigma_{\psi}(f,g)= \sigma(W(f,g);\widetilde{\psi})-n\sigma(W(f,g))
$$
and, in case dim($\mathcal{H})=\infty$
$$
\sigma_{\psi}(f,g)= \sigma^{(2)}_\G(W(f,g);\widetilde{\psi})-\sigma(W(f,g)).
$$
If $f\in C(H)$ then, by Theorem~\ref{thm:additivity}, the twisted signature $\sigma(W(f,g);\widetilde{\psi})=0$ or $\sigma^{(2)}_\G(W(f,g);\widetilde{\psi})=0$ as the case may be. If $f\in \mathcal{I}$ then by Theorem~\ref{thm:additivity}, $\sigma(W(f,g))=0$ (Meyer's cocycle vanishes). Thus, if $f\in \mathcal{C}(H)\cap \I$ then $\sigma_{\psi}(f,g)=0$.
\end{proof}

Then, as an immediate consequence of Corollary~\ref{cor:sigmavanishes}, and Proposition~\ref{prop:rhoissigdefect},

\begin{corollary}\label{cor:rhoishomo} The restriction of $\rho_\psi$ to any subgroup of $\mathcal{C}(H)\cap \I$ is a homomorphism.
\end{corollary}

\begin{proof}[Proof of Theorem~\ref{thm:additivity}] First note that part $2$ of Theorem~\ref{thm:additivity} is actually a special case of part $1$. For taking $H=F$,
note that $C(F)=\mathcal{I}$ so it will follow from part $1$ that the signature twisted by $\psi$ is zero. But in this case $F/H=0$ so the representation $\psi$ is necessarily trivial so the twisted signature is equal to the ordinary signature. Thus we need only show part $1$.

First we show that the condition that $f$ induces the identity map on $H/[H,H]$ is identical to the condition that it induces the identity on $H_1(\Sigma;\mathbb{Z}[F/H])$. Recall that, whenever an epimorphism $\phi:\pi_1(\Sigma)\to \pi_1(\Sigma)/H$ induces a coefficient system, the homology module $H_1(\Sigma;\mathbb{Z}[F/H])$ can be identified with the equivariant homology, that is the homology of the regular $F/H$-covering space of $\Sigma$ corresponding to the kernel of $\phi$, viewed as a module over $\Z[F/H]$. Since this covering space has $\pi_1$ equal to $H$, we have an identification
$$
H_1(\Sigma;\mathbb{Z}[F/H])\cong \frac{\ker\phi}{[\ker\phi,\ker\phi]}=\frac{H}{[H,H]}.
$$
(This also follows from Shapiro's lemma ~\cite[p.73]{BR}.) Hence $f$ induces the identity map on $H_1(\Sigma;\mathbb{Z}[F/H])$ if and only if it induces the identity map on $H/[H,H]$.

We now consider the proof of part $1$ of the theorem in the finite-dimensional case. The proof of part $1$ in the $\ell^{(2)}$ case is identical, with $\mathcal{U}\G$-coefficients replacing $\C^n_\psi$-coefficients.

We show that if $f$ induces the identity map on $H_1(\Sigma;\mathbb{Z}[F/H])$ then it induces the identity on $H_1(\Sigma;\psi)$. Let $\widetilde{\Sigma}$ denote the universal cover of $\Sigma$. Then, by definition,
$$
H_1(\Sigma;\psi)=H_1(\C^n\otimes_{\Z F}C_*(\widetilde{\Sigma})).
$$
But since the coefficient system factors through $F/H$ we have
$$
H_1(\Sigma;\psi)\cong H_1\left(\C^n\otimes_{\C[F/H]}\left(\C[F/H]\otimes_{\Z F}C_*(\widetilde{\Sigma})\right)\right)= H_1\left(D_*\otimes_{\C[F/H]}\C^n\right),
$$
where $D_*= \C[F/H]\otimes_{\Z F}C_*(\widetilde{\Sigma})$. Note that, by definition,
$$
H_1(D_*)= H_1(\Sigma;\C[F/H]).
$$
Now consider the commutative diagram below. We claim that the map $(id\otimes i)_*$ in the upper row is surjective.
$$
\begin{diagram}\dgARROWLENGTH=3.5em
\node{\C^n\otimes H_1\left(D_*\right)}\arrow{s,r}{id\otimes f_*}\arrow{e,t}{(id\otimes i)_*}\node{H_1(\C^n\otimes D_*)}\arrow{e,t}{\cong}\node{H_1(\Sigma;\psi)}\arrow{s,r}{f_*}\\
\node{\C^n\otimes H_1\left(D_*\right)}\arrow{e,t}{(id\otimes i)_*}\node{H_1(\C^n\otimes D_*)}\arrow{e,t}{\cong}\node{H_1(\Sigma;\psi)}
\end{diagram}
$$
Once having shown this claim, our hypothesis that $f_*$ induces the identity on $H_1(D_*)$  implies that the left-hand vertical map, $id\otimes f_*$, in the diagram  is the identity, and hence that the right-hand vertical map, $f_*$, is the identity on $H_1(\Sigma;\psi)$. To show that $(id\otimes i)_*$ is surjective, we may assume that $\Sigma$ is a complex with one zero cell and a number of $1$-cells. Lift this to an equivariant cell structure for $\widetilde{\Sigma}$. Thus $D_2=0$. Consider $\partial_1:D_1\to D_0$. Then there is an exact sequence
$$
0\to H_1(D_*)=\ker \partial_1\overset{i}{\longrightarrow}D_1\overset{\partial_1}{\longrightarrow} \text{image}~\partial_1\to 0.
$$
Since tensoring with $\C^n$ over $\C[F/H]$ is right exact, we have an exact sequence
$$
\C^n\otimes H_1(D_*) \overset{i\otimes id}{\longrightarrow}\C^n\otimes D_1\overset{id\otimes \partial_1}{\longrightarrow}\text{im}~\partial_1\otimes \C^n.
$$
Since $D_2=0$, $H_1(\C^n\otimes D_*)=\ker (id\otimes \partial_1)$. Thus
$$
\C^n\otimes H_1(D_*) \overset{(id\otimes i)_*}{\longrightarrow}H_1(\C^n \otimes D_*)
$$
is surjective. This completes the proof that $f$ induces the identity on $H_1(\Sigma;\psi)$.

Next we show that if $f$ induces the identity on $H_1(\Sigma;\psi)$ then the twisted signature $\sigma(W(f,g),\psi)$ vanishes. Following the proof of Theorem~\ref{thm:sigbounded}, we see that, in order to show that $\sigma(W(f,g),\psi)$ vanishes, it suffices to show that
$$
\rk_{\C}(\mbox{image\ }\partial_*)=0
$$
where $\partial_*$ is from the Mayer-Vietoris sequence \eqref{eq:MV} using $\C^n$-coefficients. Therefore it is sufficient to show that the composition
\begin{equation}\label{eq:mono}
H_1(\Sigma;\psi) \overset{i_*}{\longrightarrow} H_1(M_f;\psi)\overset{j_*}{\longrightarrow} H_1(N_f;\psi)
\end{equation}
is injective. There exists a Wang exact sequence for twisted homology (arising from the Serre spectral sequence for the twisted homology of the fibration $M_f\to S^1$)
$$
H_1(\Sigma;\psi)\overset{f_*-\text{id}}{\longrightarrow}
H_1(\Sigma;\psi)\overset{i_*}{\longrightarrow} H_1(M_f;\psi),
$$
which, since $f$ induces the identity on $H_1(\Sigma;\psi)$, shows that $i_*$ is a monomorphism.

Recall that $N_f$ is obtained from $M_f$ by adjoining a disjoint union of solid tori, $\partial\Sigma\times D^2$, along a disjoint union of tori, $\partial\Sigma\times S^1$. Since a solid torus is obtained from its boundary by adjoining a single $2$-handle and then a $3$-handle, $N_f$ is obtained from $M_f$ by adding a number of $2$-handles and then a number of $3$-handles. Let $\overline{N}_f$ denote the union of $M_f$ and these $2$-cells. We will show that the kernel of
$$
H_1(M_f;\psi)\overset{j_*}{\longrightarrow} H_1(\overline{N}_f;\psi)
$$
is $H_1(S^1;\psi)$ where $S^1=t=*\times S^1$. Consider the exact sequence:
\begin{equation}\label{eq:MV1}
H_1(\sqcup S^1;\psi) \overset{k_*}{\longrightarrow} H_1(M_f;\psi)\overset{j_*}{\longrightarrow} H_1(\overline{N}_f;\psi).
\end{equation}
Since $N_f$ is obtained from $M_f$  by adding two-cells along $\{t,t_1,\dots,t_m\}$, these circles constitute the $\sqcup S^1$ in the exact sequence. Note that the coefficient system is trivial on this subspace. Therefore the loops (based at $*$) $\{t,\delta_i t_i\overline{\delta}_i\}$ represent the images of the generators of $H_1(\sqcup S^1;\psi)$ under $k_*$. Recall from Subsection~\ref{subsection:3manifolds} that there are based homotopies
$$
t\sim \delta_i t_i f(\overline{\delta}_i)\sim (\delta_i t_i\overline{\delta}_i)\delta_i f(\overline{\delta}_i)\sim (\delta_i t_i \overline{\delta}_i)h_i.
$$
where, by the second hypothesis of Definition~\ref{def:C(H)}, $(f(\delta_i)\overline{\delta}_i)^{-1}=h_i$ for some $h_i\in [H,H]$. Further note that any element of $[H,H]$ represents the zero element in $H_1(M_f;\Z[F/H])$. Thus the image of $k_*$ (hence the kernel of $j_*$) is generated by the image of $t$. Now, to finish the proof that $j_*$ of sequence \eqref{eq:mono} is injective, we need only show that the image of $i_*$ from sequence \eqref{eq:mono} has trivial intersection with the image of $k_*$ ($H_1(S^1;\psi)=<t>$). Suppose that $\alpha$ is a class in the intersection. If $\pi:M_f\to S^1$ is the fibration then
$$
H_1(S^1;\psi)\overset{k_*}{\longrightarrow} H_1(M_f;\psi)\overset{\pi_*}{\longrightarrow}H_1(S^1;\psi)
$$
is the identity. Hence $\pi_*(\alpha)=\alpha$. But clearly the map
$$
H_1(\Sigma;\psi) \overset{i_*}{\longrightarrow} H_1(M_f;\psi)\overset{\pi_*}{\longrightarrow} H_1(S^1;\psi)
$$
is the zero map. Thus $\alpha=\pi_*(\alpha)=0$ as claimed. This concludes the proof of Theorem~\ref{thm:additivity}.

\end{proof}

As a further consequence of Theorem~\ref{thm:additivity}, we derive an exact sequence that generalizes the exact sequence \eqref{eq:standardexact}. Since $H$ is characteristic, any $f\in \mathcal{M}$ induces a group automorphism
$$
f_*:\frac{H}{[H,H]}\to \frac{H}{[H,H]}.
$$
But the abelian group $H/[H,H]$ may be endowed with the structure of a right (or left) $\Z[F/H]$-module via the action of $F$ on $H$ by conjugation. This module, as we observed in the first paragraph of the proof of Theorem~\ref{thm:additivity}, may be identified with the twisted homology module $H_1(\Sigma;\Z[F/H])$. If $f\in J(H)$ then $f_*$ is a \emph{module} automorphism since, for any $w\in F$ and any $h\in H$, there exists some $k\in H$ such that $f(w)=wk$. Hence
$$
f(w_*h)= f(w^{-1}hw)=f(w)^{-1}f(h)f(w)= k^{-1}w^{-1}f(h)wk\equiv w^{-1}f(h)w= w_*f(h),
$$
where the $\equiv$ means modulo $[H,H]$. Moreover, since $f$ is an orientation-preserving homeomorphism, $f_*$ is not an \emph{arbitrary} automorphism. There exists an $\Z[F/H]$-valued intersection form
$$
\lambda_H: H_1(\Sigma;\Z[F/H])\times H_1(\Sigma;\Z[F/H])\to \Z[F/H]
$$
which $f_*$ preserves ~\cite{Mi4}. Let $\mathrm{Isom}\left(H_1(\Sigma;\mathbb{Z}[F/H])\right)$ denote the group of module automorphisms of $H_1(\Sigma;\mathbb{Z}[F/H])$ that preserve $\lambda_H$.

\begin{theorem}\label{thm:new exactsequence} If $\Sigma$ has one boundary component then there is an exact sequence
\begin{eqnarray}\label{eq:nonst2}
1\to C(H)\overset{i}\longrightarrow J(H)\overset{r_\psi}\longrightarrow \mathbb{I}\mathrm{som}\left(H_1(\Sigma;\mathbb{Z}[F/H])\right),
\end{eqnarray}
and a $2$-cocycle $\tau_\psi$ on the image of $r_\psi$ such that
$$
\sigma_\psi=r_\psi^*(\tau_\psi)-n\sigma_M,
$$
if dim($\mathcal{H})=n$, and if dim($\mathcal{H})=\infty$,
$$
\sigma_\psi=r_\psi^*(\tau_\psi)-\sigma_M,
$$
where $\sigma_M$ is Meyer's cocycle restricted to $J(H)$.
\end{theorem}

\begin{remark}\label{rem:seqgeneralizesstand} Note that if $H=F$ then the exact sequence \eqref{eq:nonst2} reduces precisely to the exact sequence \eqref{eq:standardexact}.
\end{remark}

\begin{proof} The sequence is exact almost by definition. Let $\sigma^t_\psi$ denote the twisted signature $2$-cochain on $J(H)$ given by
$$
\sigma_{\psi}^t(f,g)= \sigma(V(f,g);\widetilde{\psi})=\sigma_{\psi}(f,g)+n\sigma(V(f,g))
$$
if dim($\mathcal{H})=n$, and, in case dim($\mathcal{H})=\infty$
$$
\sigma_{\psi}^t(f,g)= \sigma^{(2)}_\G(V(f,g);\widetilde{\psi})=\sigma_{\psi}(f,g)+\sigma(V(f,g)).
$$
Here we have used that $\partial \Sigma$ is connected to apply Proposition~\ref{cor:V=W} and Corollary~\ref{cor:V=W} to employ $V(f,g)$ instead of $W(f,g)$. This $2$-cochain is a $2$-cocycle on $J(H)$ since it is the sum of two $2$-cocycles. We claim that $\sigma^t_\psi$ descends to give a well-defined $2$-cocycle
$$
\widetilde{\sigma^t_\psi}:\frac{J(H)}{C(H)}\times \frac{J(H)}{C(H)}\to G,
$$
(where $G=\Z$ or $G=\R$ according as the representation is finite or infinite-dimensional). For suppose $f,g\in J(H)$ and $h\in C(H)$. Since $\sigma_{\psi}^t$ is a cocyle,
$$
\delta(\sigma^t_\psi)(f,g,h)= \sigma^t_\psi(g,h)+ -\sigma^t_\psi(fg,h)+\sigma^t_\psi(f,gh)- \sigma^t_\psi(f,g)=0.
$$
By Theorem~\ref{thm:additivity}, $\sigma^t_\psi(g,h)=0= \sigma^t_\psi(fg,h)$. Thus
$$
\sigma^t_\psi(f,gh)=\sigma^t_\psi(f,g).
$$
Hence the value of $\sigma^t_\psi$ is independent of the coset representative of $g$ in $J(H)/C(H)$. The same holds for the other variable $f$. Thus $\sigma^t_\psi$ descends to a well-defined $2$-cocycle,  denoted $\tau_\psi$ , on the image of $r_\psi$ such that $r_\psi^*(\tau_\psi)=\sigma^t_\psi$. Moreover, using Proposition~\ref{prop:V=W}, $\sigma(V(f,g))=\sigma(W(f,g))=\sigma_M(f,g)$ (the Meyer cocycle),
$$
r_\psi^*(\tau_\psi)-n\sigma_M=\sigma^t_\psi -n\sigma_M=\sigma_\psi
$$
if dim($\mathcal{H})=n$, whereas, if dim($\mathcal{H})=\infty$, then
$$
r_\psi^*(\tau_\psi)-\sigma_M=\sigma_\psi.
$$
\end{proof}

\end{subsection}

\section{Examples and Calculations}\label{sec:examples}

In this section we perform calculations in one of the simplest non-classical cases in order to exhibit the complexity of the higher-order signature cocycles and $\rho$-invariants. 

Specifically, let $\Sigma=\Sigma_{g,1}$ where $g\geq 2$ and $H=[F,F]$ so $J(H)=\mathcal{I}$. For any norm $1$ complex number $\omega$, we can define a higher-order $\rho$-invariant $\rho_\omega=\rho_{\psi_\omega}$ as follows. Choose $\psi_\omega:F/H\to U(1)$ as the composition
$$
F/H\cong H_1(\Sigma;\mathbb{Z})\cong \Z^{2g}\overset{\pi}{\longrightarrow}S^1\equiv U(1),
$$
where, for each $\pi$ sends every element of a fixed basis to $\omega$. To be more precise, let $x_i$ and $y_i$ be the curves on the surface $\Sigma_{g,1}$ as indicated in Figure~\ref{gens}.  These generate $\pi_1(\Sigma_{g,1},\star)$.

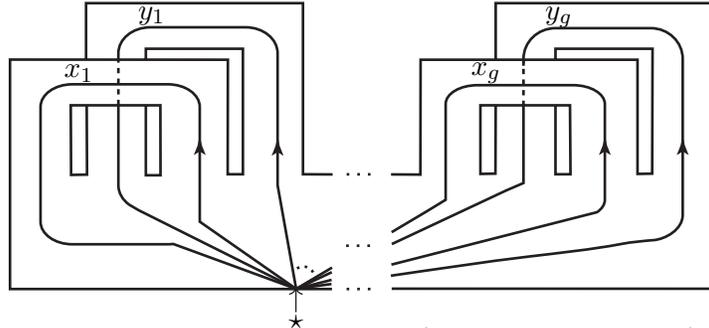
\begin{figure}[htb]
\begin{picture}(268,110)(0,0)
\put(24,82){$x_1$}
\put( 52,104){$y_1$}
\put(178,83){$x_g$}
\put(206,104){$y_g$}
\put(109,-13){$\star$}
\put(109,-5){$\uparrow$}
\put(3,1){\includegraphics{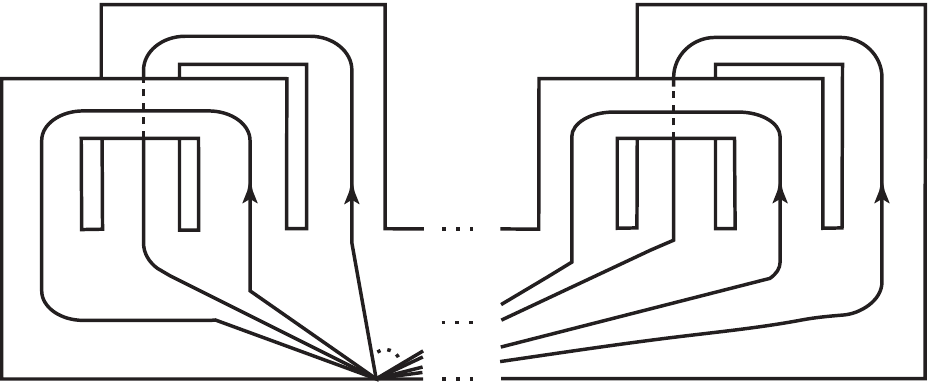}}
\end{picture}
\caption{The curves $x_i$ and $y_i$ generate the fundamental group of the punctured surface $\Sigma_{g,1}$.}\label{gens}
\end{figure}

For each $\omega \in \mathbb{C}$ such that $||\omega||=1$, let $\psi_\omega : H_1(\Sigma_{g,1}) \rightarrow U(1)$ be the representation that sends each $x_j$ and $y_j$ to $\omega$.  Define $\rho_\omega(f) := \rho(f,  \psi_\omega \circ \pi)$ for any $f\in \mathcal{I}(\Sigma_{g,1})$.

We introduce some examples in $\mathcal{K}_g$ on which we can calculate $\rho_\omega$.
For each $m\geq 1$ and $n\geq 0$, let $\alpha$ and $\beta(m,n)$ be the curves on $\Sigma_{g,1}$ as indicated in Figure~\ref{curves} where $2m$ and $2n$ are the number of times $\beta(m,n)$ passes over the ``first $1$-handle'' and ``third $1$-handle'' respectively. (In the figure, if you ignore the ellipses, $n=m=3$.) Even though figure shows a genus $2$ surface, the reader should imagine that the other $2g-4$-handles of $\Sigma$ are adjoined, say, on the left-hand side of the figure. They will play no role in the computations to come, since the homeomorphisms we consider will be supported in the genus two subsurface pictured in Figure~\ref{curves}. Thus  the following computations suffice for any $g\geq 2$. In Proposition~\ref{prop:subsurfacerestrict}, this paradigm (about the equality of the $\rho$-invariants computed from a subsurface with those computed from the super-surface) is formalized. For our convenience, we will often write drop the $m$ and $n$ from the notation and write $\beta$ instead of $\beta(m,n)$.  Let $x=x_1$, $y=y_1$, $z=x_2$, and $w=y_2$. Then, up to conjugation and a choice of orientation, $\alpha$ and $\beta$ represent the homotopy classes $[z,w]$ and $[z^n w^{-1}, x^{-m} y^{-1}] [x^{-1},y]$ respectively. Since $\alpha$ and $\beta$ are bounding curves (either by direct observation or by observing that they are null-homologous simple closed curves), we have that  $D_\alpha, D_{\beta} \in \mathcal{J}(3)=\mathcal{K}_g$ where $D_\alpha$ and $D_{\beta}$ are the Dehn twists about $\alpha$ and $\beta$ respectively.
For each $m\geq 1$, $n\geq 0$ and $N \in \Z$, define $f_{(m,n,N)} := (D_\alpha \circ D_{\beta(m,n)})^{N+1} \in \mathcal{K}_g$.

\begin{figure}[htb]
\begin{picture}(398,293)(0,0)
\put(0,0){\includegraphics[scale=0.45]{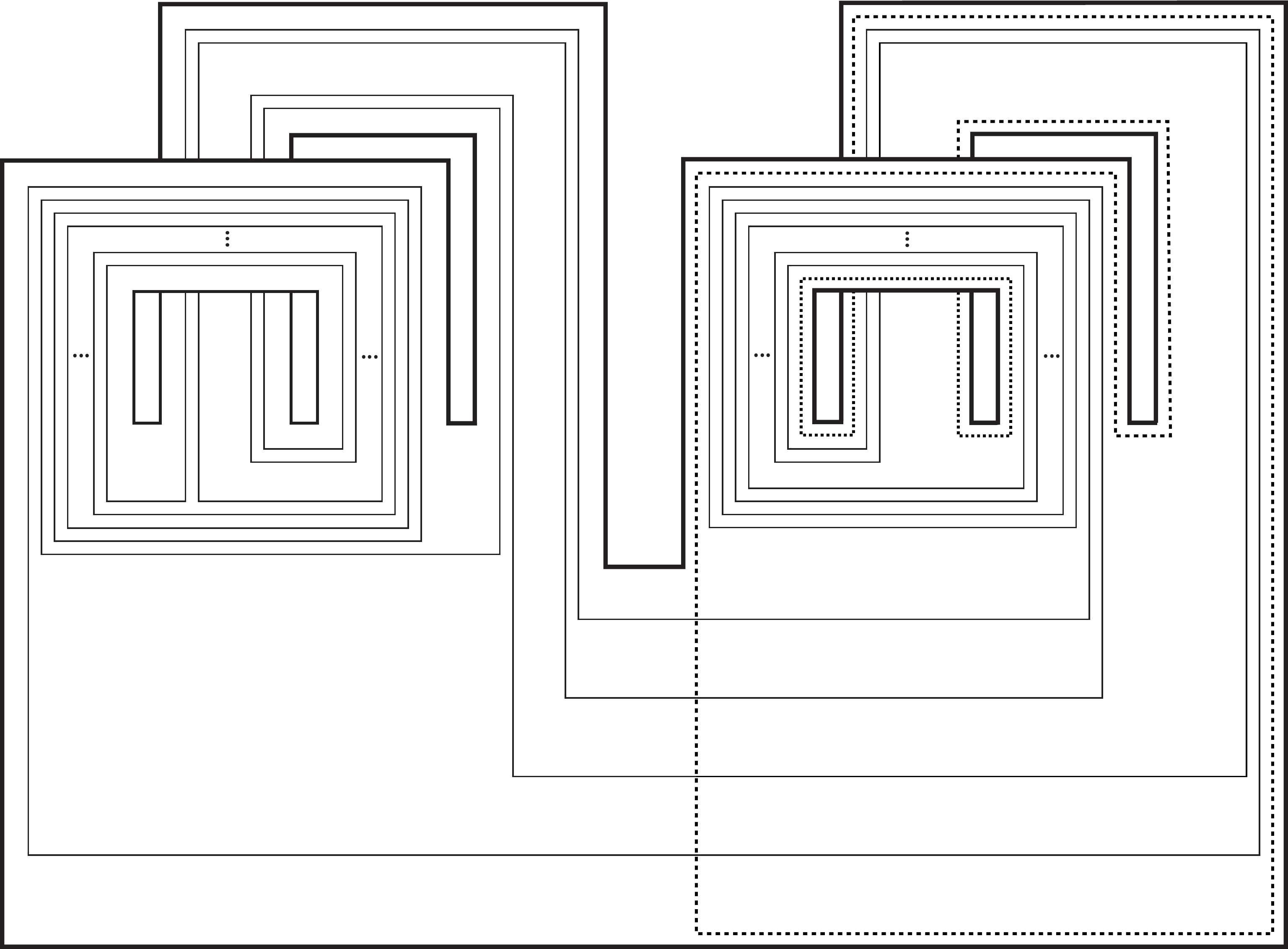}}
\put(219,118){$\alpha$}
\put(197,106){$\beta$}
\end{picture}
\caption{The curves $\alpha$ and $\beta$.}\label{curves}
\end{figure}

\begin{lemma}\label{lem:maincalc} Let $m\geq 1$, $n, N\geq 0$, $G_{(m,n)}(t) = (t^{(n-1)}-1)(t^{-(m+1)}-1)$ and $\omega \in \mathbb{C}$ have norm $1$ with $\omega \neq 1$.  Then  $\rho_\omega(f_{(m,n,N)})+2(N+1)$ is equal to the signature of the $2N \times 2N$ hermitian matrix
\begin{equation} C_{(m,n,N)}(\omega) :=  \left( \begin{array}{cc} A & \overline{G_{(m,n)}(\omega)}B^T  \\ G_{(m,n)}(\omega) B & A\end{array} \right)  \label{int_matrix}
\end{equation}
where
\begin{equation}
A= \left( \begin{array}{cccccc} 2 & -1 & 0 & \cdots & 0 & 0 \\
-1 & 2 & -1 & & \\
0 & -1 & 2 & &  \\
\vdots & & & \ddots &\\
0  &&&& 2 & -1 \\
0 &&&& -1 & 2
\end{array} \right),
\end{equation}
is the $N \times N$ matrix with a $2$ in all the diagonal entries and a $-1$ in all the super- and sub-diagonal entries, and
\begin{equation}
B = \left( \begin{array}{cccccc} -1 & 0 & 0 & \cdots & 0 & 0 \\
1 & -1 & 0 & & \\
0 & 1 & -1 & &  \\
\vdots & & & \ddots &\\
0  &&&& -1 & 0 \\
0 &&&& 1 & -1
\end{array} \right),
\end{equation}
is the $N \times N$ matrix with a $-1$ in all the diagonal entries and a $1$ in all the sub-diagonal entries.\label{thematrix}
\end{lemma}
\begin{proof}  First, we claim that $N_\text{id} \cong \#_{2g} S^1 \times S^2$ where $\text{id} : \Sigma_{g,1} \rightarrow \Sigma_{g,1}$ and that the inclusion map $\Sigma_{g,1} \times \{0\} \rightarrow N_\text{id}$ induces an isomorphism on $\pi_1(-,\star)$.  To see this first observe that $N_{id}=\partial (\Sigma\times D^2$). Then note that since $\Sigma$ may be built from a single $0$-handle and $2g$ $1$-handles,  $\Sigma\times D^2$ may be built from a $4$-ball and $2g$ ($4$-dimensional)  $1$-handles. Any such (orientable) manifold is homeomorphic to $\natural_{2g} S^1\times B^3$ (here $\natural$ denotes boundary connected sum).

Fix the integers $N, n$, $m$, let $f=f_{(m,n,N)}$ and consider the following set of $2N+2$ curves in $\Sigma \times \left[0,1\right] \subset N_{\text{id}}$
$$\mathcal{S} = \{  \beta \times \{{2i}/{(2N+2)}\}, \alpha \times \{(2i + 1)/(2N+2)\}~|~ 0\leq i \leq N \}.$$
Let $X$ be the $4$-manifold obtained by attaching $2N+2$ $2$-handles to $N_{\text{id}} \times I$ along the curves in $\mathcal{S} \times \{1\} \subset N_{\text{id}} \times \{1\}$, each with $+1$ framing.  Then $\partial X = \overline{N}_{\text{id}} \sqcup N_f.$ This statement is well-known ~\cite[proof of Theorem 2]{Lick1}\cite[p.277]{R}. For the reader who is unfamiliar with it, note that it suffices to show that adding a single $2$-handle with $+1$-framing yields a new ``top'' boundary component that still fibers over the circle but whose monodromy is altered by a Dehn twist along the attaching circle of the handle. In turn, to prove this latter fact, it suffices to prove it for the product fibration of an annulus over $S^1$ (since the handles are added along a thickened annulus).

Since the curves $\alpha$ and $\beta$ are null-homologous in $\Sigma$, the $2$-handles are attached to circles that are null-homologous in $N_{id}$. Thus, using equation \eqref{eq:pi1N}, the inclusion map  $i_1: N_{id} \rightarrow X$ induces an isomorphism on $H_1(-;\mathbb{Z})$. It follows that the inclusion map $i_1: N_{f} \rightarrow X$ induces an isomorphism on $H_1(-;\mathbb{Z})$. Hence we can extend $\psi_\omega \circ \pi : N_{f} \rightarrow U(1)$ to $\Phi : \pi_1(X) \rightarrow U(1)$ in the obvious way so that $\Phi_{| N_\text{id}}=  \psi_\omega \circ \pi $. Recall that $N_\text{id}$ is the boundary of the boundary-connected-sum of $2g$ copies of $S^1 \times B^3$, denoted $E$, wherein the inclusion map induces an isomorphism on $\pi_1(-)$.  Let $W = X \cup \overline{E}$.  Since the inclusion map $N_{\text{id}} \rightarrow E$ induces an isomorphism on $\pi_1(-)$, we can extend $\Phi: \pi_1(X) \rightarrow U(1)$ to $\Phi : \pi_1(W) \rightarrow U(1)$. Thus,
\begin{equation}\rho_\omega(f) = \sigma(W,\mathbb{C}_\Phi) - \sigma_0(W) \label{eq}
\end{equation}
where  $\sigma(W,\mathbb{C}_\Phi)$ is the twisted signature of $W$ (twisted by $\Phi$) and $\sigma_0(W)$ is the ordinary signature of $W$.

We first consider $H_2(W)$.  Since each curve, $\alpha$ and $\beta$, bounds a punctured torus in $\Sigma$, $H_2(W) \cong \mathbb{Z}^{2N+2}$; it is generated by the tori obtained by capping off these punctured tori by disks that are the cores of the attached 2-handles.  Note that the tori are all disjointly embedded and they have self-intersection $+1$. Thus $\sigma_0(W)=2N+2$.

Next we consider $H_2(W;\mathbb{C}_\Phi)$. Let $Y_1$ be the $4$-manifold obtained attaching two $2$-handles to $E$ along $\beta \times \{0\}$ and $\alpha \times \{1/(2N)\} \subset N_\text{id}=\partial{E}$. We claim that $H_2(Y_1;\mathbb{C}_\Phi)=0$. This involves a calculation using Fox calculus. Since $H_2(E;\mathbb{C}_\Phi)=0$, $H_2(Y_1;\mathbb{C}_\Phi)=0$ if and only if $\beta \times \{0\}$ and $\alpha \times \{1/(2N)\}$ are linearly independent in $H_1(E;\mathbb{C}_\Phi)$.  Since $H_1(E;\mathbb{C}_\Phi) \subset H_1(E,\star;\mathbb{C}_\Phi)$, it suffices to consider $\beta \times \{0\}$ and $\alpha \times \{1/(2N)\}$ in $H_1(E,\star;\mathbb{C}_\Phi)$.   We denote $x_1, y_1, x_2, y_2$ by $x,y,z,w$ respectively and view these as the generators of $\pi_1(E)$.
Let $\tilde{\star}$ be a lift of $\star$ to the universal cover of $E$ and $\tilde{x},\tilde{y},\tilde{z},\tilde{w}$ be lifts of $x,y,z,w$ starting at $\tilde{\star}$ respectively.   Then $H_1(E,\star;\mathbb{C}_\Phi)\cong \mathbb{C}^{4}$ is generated by $\{\mathbf{x} = \tilde{x} \otimes 1,  \mathbf{y} = \tilde{y} \otimes 1, \mathbf{z}=\tilde{z} \otimes 1, \mathbf{w} = \tilde{w}\otimes 1\}$.

Let $\gamma$ be a path on $\Sigma$ that goes ``straight'' from $\star$ to the ``top'' intersection of $\alpha$ and $\beta$.  We will use $\gamma$ along with ``straight line'' paths in the $[0,1]$ direction of $\Sigma \times [0,1] \subset N_\text{id}$ to base the curves in $\mathcal{S}$.  Orient $\alpha$ and $\beta$ so that the arrows on their rightmost vertical segments are pointing upward.  With these conventions, $\alpha=z^{-1}[z,w]z$ and $\beta=[y,x^{-1}][(yx^m)^{-1},z^nw^{-1}]$ in $\pi_1(E)$. We calculate the Fox derivatives of $\alpha$ and $\beta$ with respect to $x,y,z,w$.
  \begin{align*} \frac{\partial \alpha}{\partial x} &= \frac{\partial \alpha}{\partial y}=0\\
  \frac{\partial \alpha}{\partial z}&=wz^{-1}(w^{-1}-1) \\
  \frac{\partial \alpha}{\partial w}&=1-wz^{-1}w^{-1}\\
  \frac{\partial \beta}{\partial x}&=  yx^{-1}(y^{-1} -1)+ [y,x^{-1}] x^{-m}(y^{-1} z^n w^{-1} y-1)(1+ \cdots +  x^{m-1})\\
   \frac{\partial \beta}{\partial y}&=1-yx^{-1} y^{-1} + [y,x^{-1}] x^{-m}y^{-1}(z^nw^{-1}-1)\\
   \frac{\partial \beta}{\partial z}&=[y,x^{-1}] x^{-m}y^{-1}(1-z^n w^{-1} y x^m w z^{-n})(1 + z + \cdots + z^{n-1})\\
   \frac{\partial \beta}{\partial w}&= [y,x^{-1}] x^{-m} y^{-1} z^n w^{-1} (yx^m -1) \\
  \end{align*}
 Setting $x=y=z=w=\omega$, we can write $\alpha$ and $\beta$ as elements of $H_1(E,\star;\mathbb{C}_\Phi)$.
  \begin{align*}
  \alpha =& (\omega^{-1} -1) \mathbf{z} + (1-\omega^{-1})\mathbf{w}\\
  \beta =& \left((\omega^{-1} -1)+ \omega^{-m}(\omega^{n-1} -1)(1+ \omega + \cdots + \omega^{m-1})\right)\mathbf{x}+ \left( 1-\omega^{-1} + \omega^{-(m+1)}(\omega^{n-1}-1) \right)\mathbf{y} +\\
  &+ \left( (\omega^{-m-1}-1)(1+\omega + \cdots + \omega^{n-1}) \right)\mathbf{z} + \left( \omega^{n-m-2}(\omega^{m+1} -1) \right)\mathbf{w}
  \end{align*}

  Since $\omega\neq 1$, $\alpha \neq 0$.  We now show that $\beta$ is not a multiple of $\alpha$ which will complete the proof that $H_2(Y_1;\mathbb{C}_\Phi)=0$.
  Suppose $\beta = \lambda \alpha$ then we have the following system of equations.
 \begin{align}  (1-\omega)(\omega^{-1} -1) &=  (\omega^{n-1} -1)(\omega^{-m}-1) \\
\omega^{-1} -1 &=  \omega^{-(m+1)}(\omega^{n-1}-1) \\
(\omega^{-(m+1)}-1)(\omega^n -1) &= \lambda (\omega^{-1} -1)(\omega-1) \\
\omega^{n-m-2}(\omega^{m+1} -1) &= \lambda (1-\omega^{-1})
 \end{align}
 Taking the norm of both sides of $(19)$, we see that $||\omega^{-1}-1||=||\omega^{n-1}-1||$. Since $\omega^{-1}$ and $\omega^{n-1}$ are on the unit circle, this implies that $\omega^{n-1}=\omega^{-1}$ or $\omega^{n-1} = \omega$.

 We first consider the case that $\omega^{n-1}=\omega^{-1}$.  In this case, $\omega^n=1$ so using equation $(20)$, we see that $\lambda (\omega^{-1} -1)(\omega-1) =0$.  Since $\omega \neq 1$, we have that $\lambda=0$.  By equation $(21)$, $\omega^{m+1}=1$. However, this cannot happen since substituting $\omega^n=1$ and $\omega^{m+1}=1$ in equation $(18)$ gives $-(\omega-1)(\omega^{-1}-1) = (\omega^{-1}-1)(\omega -1)$.

 We now consider the case that $\omega^{n-1}=\omega$. Substituting this into equation (19) and multiplying both sides by $\omega$ gives $(1-\omega) =\omega^{-m} (\omega-1)$. Since $\omega \neq 1$, we must have that $\omega^{-m}=-1$. With the substitutions $\omega^{n-1}=\omega$ and $\omega^{-m}=-1$, equation (18) becomes $(1-\omega)(\omega^{-1} -1) =  -2(\omega-1)$.  However, this would imply that $\omega^{-1}=3$ which cannot happen since $\omega$ is on the unit circle. This completes the proof that $\alpha$ and $\beta$ are linearly independent and hence $H_2(Y_1;\mathbb{C}_\Phi)=0$.

Now we return to our calculation of $H_2(W;\mathbb{C}_\Phi)$.  Let $U$ be the region in Figure~\ref{dashed} enclosed by the dashed lines.  A picture of the attaching curves (when $N=3$) in $U \times I$ is shown in Figure~\ref{attach}.  The attaching curves outside of $U \times I$ are ``parallel'' to the original $\alpha$ or $\beta$.

 \begin{figure}[htb]
 \begin{picture}(354,260)(0,0)
\put(180,15){$U$}
\put(0,0){\includegraphics[scale=.40]{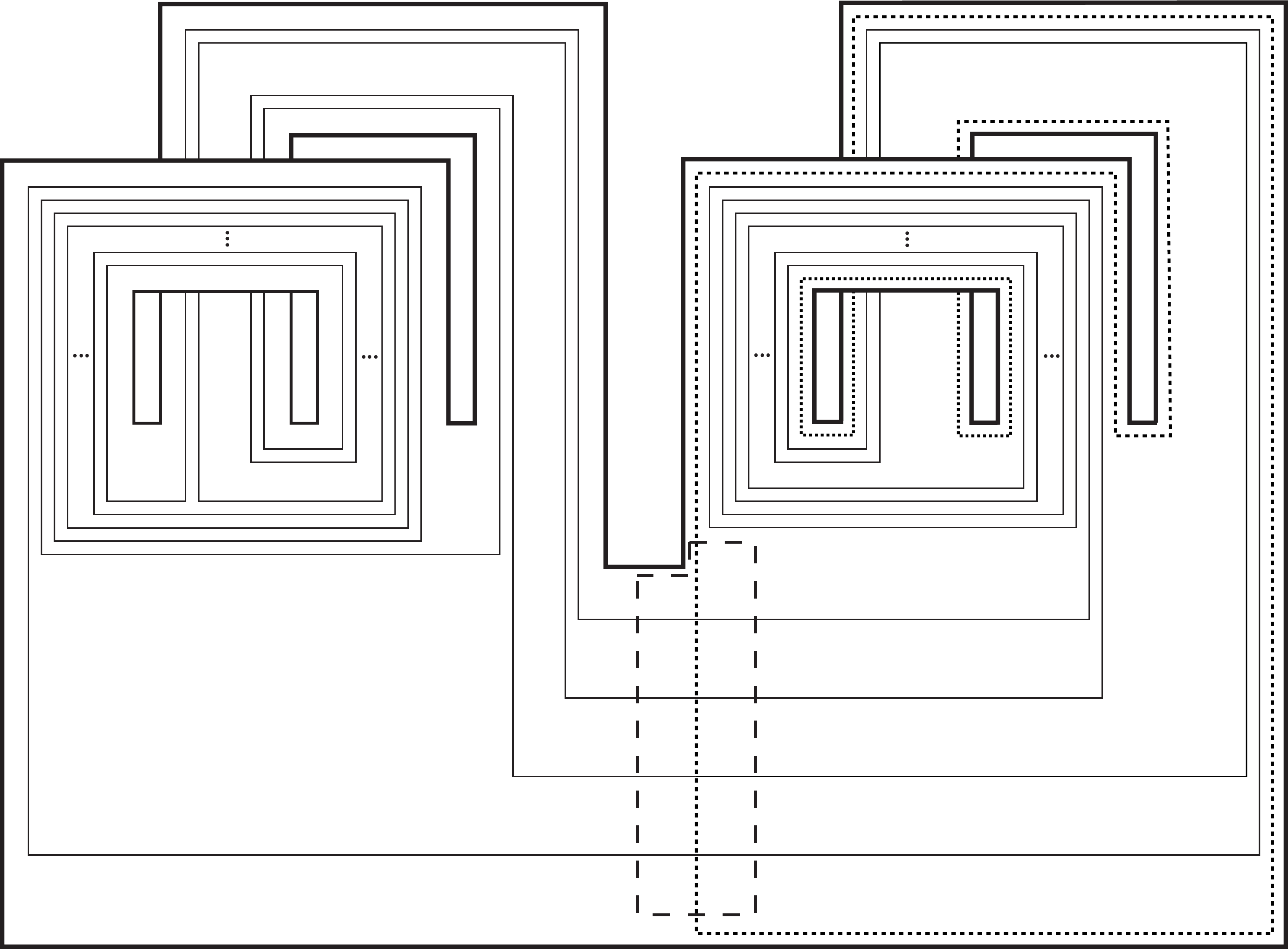}}
\end{picture}
\caption{The region $U$ in $\Sigma$} \label{dashed}
\end{figure}

 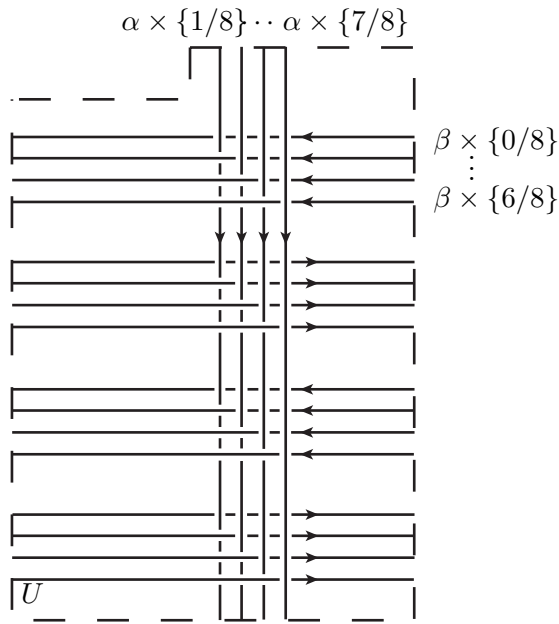
\begin{figure}[htb]
\begin{picture}(153,230)(0,0)
\put(24,8){$U$}
\put(180,180){$\beta \times \{0/8\}$}
\put(193,168){$\vdots$}
\put(180,158){$\beta \times \{6/8\}$}
\put(62,225){$\alpha \times \{1/8\}$}
\put(107,225){$\cdots$}
\put(123,225){$\alpha \times \{7/8\}$}
\put(20,1){\includegraphics{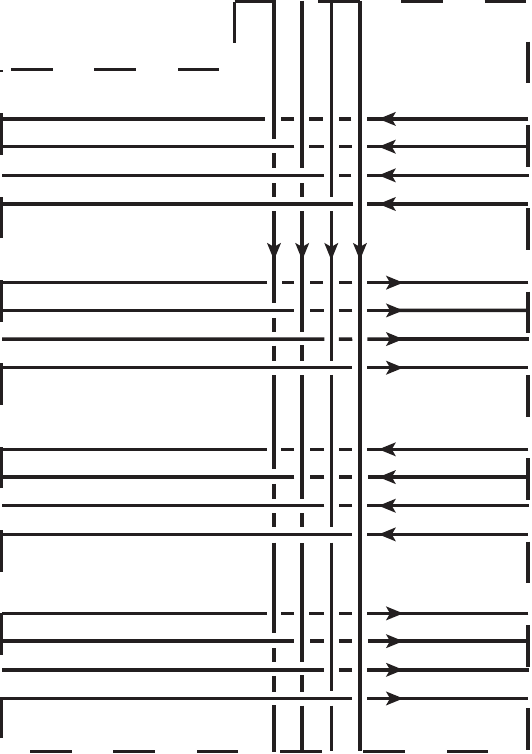}}
\end{picture}
\caption{Attaching curves when $N=3$} \label{attach}
\end{figure}

 Slide the handle attached along
$$
\alpha \times \{{(2N+1)}/{(2N+2)}\}
$$
over the handle attached along $\alpha \times \{{(2N-1)}/{(2N+2)}\}$ and call the resulting attaching curve $\alpha^\ast_{N}$.  Then slide the handle attached along $\beta \times \{{(2N)}/{(2N+2)}\}$ over the handle attached along $\beta \times \{{(2N-2)}/{(2N+2)}\}$ and call the resulting attaching curve $\beta^\ast_{N}$.  Continue this; for $i$ from $1$ to $N-1$, slide the handle attached along $\alpha \times \{{(2N-2i+1)}/{(2N+2)}\}$ (respectively $\beta \times \{{(2N-2i)}/{(2N+2)}\}$)  over the handle attached along $\alpha \times \{{(2N-2i-1)}/{(2N+2)}\}$ (respectively $\beta \times \{{(2N-2i-2)}/{(2N+2)}\}$) and call the resulting attaching curve $\alpha^\ast_{N-i}$ (respectively $\beta^\ast_{N-i}$).  A local picture of the new attaching curves  is shown in Figure~\ref{handle_slide}.

 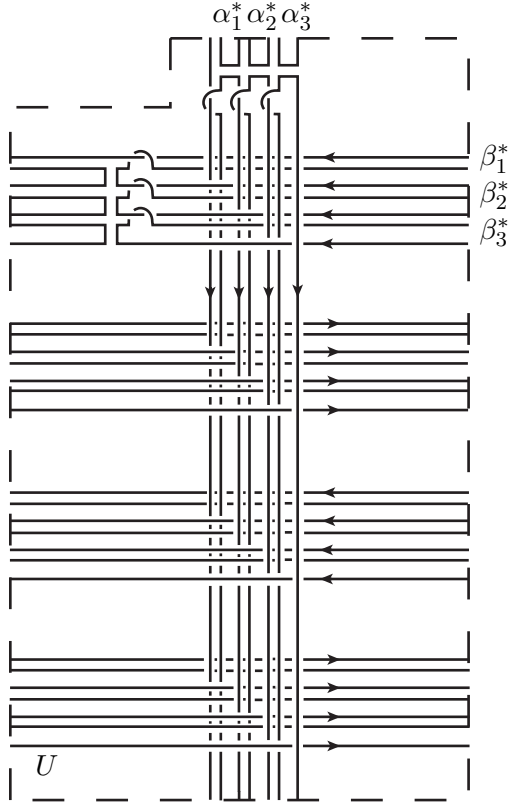
\begin{figure}[htb]
\begin{picture}(174,300)(0,0)
\put(10,10){$U$}
\put(178,213){$\beta^\ast_{3}$}
\put(178,227){$\beta^\ast_{2}$}
\put(178,241){$\beta^\ast_{1}$}
\put(103,295){$\alpha^\ast_{3}$}
\put(90,295){$\alpha^\ast_{2}$}
\put(77,295){$\alpha^\ast_{1}$}
\put(0,0){\includegraphics{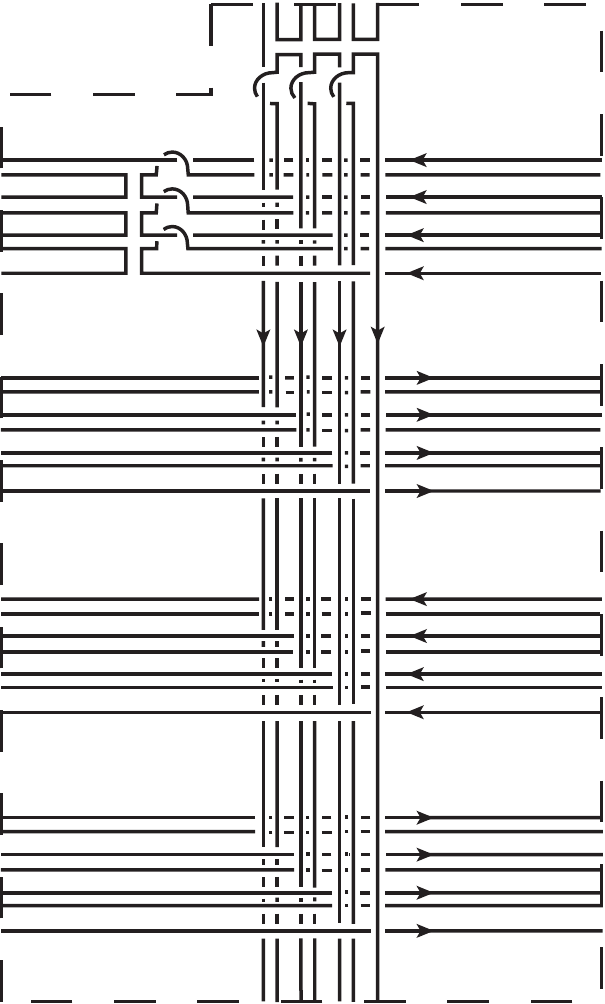}}
\end{picture}
\caption{Attaching curves after handle slides when $N=3$}\label{handle_slide}
\end{figure}

Note that each $\alpha_i^\ast$ (respectively $\beta_i^\ast$), oriented as described, bounds an obvious oriented embedded disk $D_{\alpha,i}$ (respectively $D_{\beta,i}$) in $Y_1$ for $1 \leq i \leq N$.  For each $1\leq i \leq N$, let $F_{\alpha,i}$ (respectively $F_{\beta,i}$) be the oriented embedded $2$-sphere obtained by gluing the core of the $2$-handle attached along $\alpha_i^\ast$ (respectively $\beta_i^\ast$) to $D_{\alpha,i}$ (respectively $D_{\beta,i}$) so that the orientation of $F_{\alpha,i}$ (respectively $F_{\beta,i}$) agrees with the orientation on $D_{\alpha,i}$ (respectively $D_{\beta,i}$). Therefore $H_2(W;\mathbb{C}_\Phi)\cong \mathbb{C}^{2N}$ and has as an ordered basis $F_{\alpha,1},\dots, F_{\alpha,N},F_{\beta,1}, \dots, F_{\beta,N}$.
Using this basis, it is straightforward to check that the intersection form on $H_2(W;\mathbb{C}_\Phi)$ is given by the matrix in equation \eqref{int_matrix}.  For example, consider $F_{\alpha,1} \cdot F_{\beta,1}$.  After making the surfaces transverse, there are $4$ intersection points ($2$ positive and $2$ negative).  Taking into account the weightings from $\pi_1(W)$, we see that the equivariant intersection number is $-1 + z^n w^{-1} - z^n w^{-1} x^{-m} y^{-1} + z^n w^{-1} x^{-m} y^{-1} w  z^n  \in \Z [\pi_1(W)]$.
Therefore $F_{\alpha,1} \cdot F_{\beta,1} = -1 + \omega^{n-1} - \omega^{(n-1) - (m+1)} + \omega^{-(m+1)} = -G_{(m,n)}(\omega).$
\end{proof}

We interrupt our discussion to point out an interesting connection to signatures of Lefschetz fibrations:
\begin{proposition}\label{prop:Lefschetz} Given $\Sigma_{g,m}$, suppose that $D_1,\dots,D_n$ are positive Dehn twists along null-homologous circles. Then, for any unitary representation $\psi$ of $F/[F,F]\equiv H_1(\Sigma;\Z)$,
$$
\rho_\psi(D_n\circ\dots\circ D_1)= \sigma(Y,\psi)-\sigma(Y)
$$
where $Y$ is the Lefschetz fibration over the $2$-disk with generic fiber $\Sigma$ and with $n$ singular fibers whose monodromies are $D_1,\dots,D_n$.
\end{proposition}
\begin{proof} Note that the construction of the $4$-manifold $W$ in the preceding proof will produce, in this greater generality, a null-bordism for $N_{D_n\circ\dots\circ D_1}$. Thus by Theorem~\ref{thm:APS},
$$
\rho_\psi(D_n\circ\dots\circ D_1)= \sigma(W,\psi)-\sigma(W).
$$
Then it is only necessary to identify $W$ with $Y$.  For this it is known that $W$ is obtained from $E\cong \Sigma\times D^2$, by adding two handles along separating curves.
For details see, for example, ~\cite[Section 8.2]{GomStip}.
\end{proof}

\begin{lemma}\label{calc_sig}Let $r\geq 2$ and $N_0 \geq 0$ be integers.  Then

$$\mathrm{signature}(C_{(r-1,r+1,2N_0)}(\omega)) =  \left\{ \begin{array}{cc} 4N_0 & \text{if } \omega^r =  1  \\ 0 & \text{if } \omega^r = \pm i  \end{array} \right. $$
\end{lemma}

\begin{proof}
Let $m=r-1$ and $n=r+1$ and $N=2N_0$.  Since $|| \omega || =1$, we have $G_{(m,n)}(\omega) = || \omega^r -1 ||^2$.  So when $\omega^r =1$, $G_{(m,n)}(\omega)=0$ so $C_{(r-1,r+1,2N_0)}(\omega))$ is a block sum of $2$ copies of $A$.   It is easily shown that $A$ has signature $N=2N_0$ (or note that $A=B^*B$) so $\mathrm{signature}(C_{(r-1,r+1,2N_0)}(\omega))=2N = 4N_0$.

We now consider the case when $\omega^r =\pm i$.  In this case, $G_{(m,n)}(\omega) = 2$.  By adding rows/columns $1$ through $N$ to rows/columns $N+1$ through $2N$ respectively, we see that $C_{(m,n,N)}(\omega)$ is congruent to the following matrix
\begin{equation} \left( \begin{array}{cc} A &  A + 2 B^T  \\  A + 2 B & 2 A + 2 B^T + 2 B\end{array} \right) =  \left( \begin{array}{cc} A &  A + 2 B^T  \\  A + 2 B & 0\end{array} \right). \label{nonsing}
\end{equation}
Let $C^\prime$ be the matrix in equation \eqref{nonsing}. We will show that $C^\prime$ is non-singular whenever $N$ is even.  Since $C^\prime$ has a half block of zeros in the lower right corner, it follows that it has signature $0$ which will complete the proof.

First note that $\mathrm{det}(C^\prime) = -\mathrm{det}(A+ 2B)^2$ so it suffices to show that $\mathrm{det}(A+ 2B) \neq 0$. We will prove $\mathrm{det}(A+ 2B) = 1$ by induction on even $N$.
$$A + 2B = \left( \begin{array}{cccccc} 0 & 1 & 0 & \cdots & 0 & 0 \\
-1 & 0 & 1 & & \\
0 & -1 & 0 & &  \\
\vdots & & & \ddots &\\
0  &&&& 0 & 1 \\
0 &&&& -1 & 0
\end{array} \right)_{N \times N}$$
When $N=2$, $\det(A + 2B)=1$.  Suppose $\mathrm{det}(A+2B_{N\times N}) =1$ for some even $N$.  We expand the determinant  twice (first along the first column and then along the first row) to get the inductive formula:  $\mathrm{det}(A+2B_{(N+2)\times (N+2)}) = \mathrm{det}(A+2B_{N \times N})$.  Hence $\mathrm{det}(A+2B_{(N+2)\times (N+2)})=1$.

\end{proof}

For $k\geq 1$, let $\omega_k := e^{2\pi i/4^k}$ and set $\rho_k := \rho_{\omega_k}$.
We will show that the set of $\rho_k$ generates an infinitely generated subset of $\widehat{Q}(\mathcal{J}(3))$.

\begin{theorem}\label{thm:infgenquasi} For $g\geq 2$, $\{\rho_k\}$ is a linearly independent subset of $\widehat{Q}(\mathcal{J}(3))$.
\end{theorem}
\begin{proof} To prove this, we must show that no non-trivial linear combination of the $\rho_k$ is a bounded function.  Let $k_1, \dots, k_l$ be an increasing sequence of $l$ positive integers.  Suppose that
$$\displaystyle\sum_{i=1}^l a_i \rho_{k_i} = \delta$$
where $a_i\neq 0$, $|\delta(g)| \leq M$ for all $g \in \mathcal{J}(3)$ where $M$ is a constant.  Consider $f_{(m,n,N)} = (D_\alpha \circ D_{\beta(m,n)})^{N+1} $,  as defined in the paragraph directly preceding Lemma~\ref{thematrix}.
Since $$\omega_k^{4^{j}} = \left\{ \begin{array}{cc} i & \text{if } j=k-1  \\ 1 & \text{if } j \geq k  \end{array} \right. ,$$by Lemmas~\ref{thematrix} and \ref{calc_sig},

$$ \rho_{k}(f_{(4^j -1, 4^j +1, 2N_0)})  = \left\{ \begin{array}{cc}  -2(2N_0 +1) & \text{if } j=k-1  \\ -2  & \text{if } j\geq k  \\ \end{array} \right. .$$
Therefore, when $j=k_1-1$, we have
\begin{align*} M &\geq \left|\delta(f_{(4^{k_1-1} -1, 4^{k_1-1} +1, 2N_0)})\right|  \\ &= \left|\displaystyle\sum_{i=1}^l a_i \rho_{k_i} (f_{(4^{k_1-1} -1, 4^{k_1-1} +1, 2N_0)}\right|\\
&= \left| a_1 2(2N_0+1) + \displaystyle\sum_{i=2}^l 2 a_i \right|
\end{align*}
Dividing by $2\left| a_1 \right|$ we see that $\left| (2N_0+1) + \displaystyle\sum_{i=2}^l 2 a_i \right| \leq M/(2\left| a_1\right|)$. However, since all the $a_i$ and $M$ are fixed and $N_0$ can be chosen to be arbitrarily large, this is a contradiction.

\end{proof}

Note that in the above proof, $\beta$ depends on the linear combination. We have actually shown that any particular non-trivial linear combination of the $\rho_k$ is an  unbounded function on the cyclic subgroup generated by $D_\alpha\circ D_\beta$, for suitably chosen $\beta$.

\begin{theorem}\label{thm:infgenbounded} For $g\geq 2$, $\{\delta(\rho_k)\}$ is a linearly independent subset of $H^2_b(\mathcal{J}(3);\mathbb{R})$, the second bounded cohomology of $\mathcal{J}(3)$.
\end{theorem}

\begin{proof} Recall the key exact sequence:
$$
0\to H^1(\mathcal{J}(3);\mathbb{R})\to \widehat{Q}(\mathcal{J}(3))\overset{\delta}\longrightarrow H^2_b(\mathcal{J}(3);\mathbb{R})\to H^2(\mathcal{J}(3);\mathbb{R}).
$$
From this we deduce that we must show that no non-trivial linear combination of the $\rho_k$ is equal to a homomorphism plus a bounded function. As above, suppose that
$$\displaystyle\sum_{i=1}^l a_i \rho_{k_i} =\phi+ \delta$$
where $a_i\neq 0$, $\phi$ is a homomorphism and $\delta$ is a bounded function.

\begin{lemma}\label{lem:boundedoncyclic} Let $D$ denote $D_\alpha$ or $D_\beta$  for any $\alpha, \beta$. For each $k$, $\{\rho_k(D^M)~|~ M\in \Z\}$  is a bounded set.
\end{lemma}

First we will show that Lemma~\ref{lem:boundedoncyclic} implies Theorem~\ref{thm:infgenbounded}. It follows directly from the lemma that
$$\displaystyle\sum_{i=1}^l a_i \rho_{k_i}(D^M)$$
is a bounded set (only $M$ is varying here). On the other hand
$$\phi(D^M)+ \delta(D^M)=M\phi(D)+ \delta(D^M)$$
is an unbounded set unless $\phi(D)=0$. Therefore we may assume that $\phi(D_\alpha)=0$ and $\phi(D_\beta)=0$ and hence, since $\phi$ is a homomorphism, that $\phi$ vanishes on the subgroup generated by $D_\alpha$ and $D_\beta$. It would follow that, on the subgroup generated by $D_\alpha$ and $D_\beta$,
$$\displaystyle\sum_{i=1}^l a_i \rho_{k_i} =\delta,$$
which is a bounded function. In particular it is a bounded function on the cyclic subgroup generated by $D_\alpha\circ D_\beta$. However, after choosing $\beta$ suitably, this contradicts the proof of Theorem~\ref{thm:infgenquasi}.

\begin{proof}[Proof of Lemma~\ref{lem:boundedoncyclic}] In brief, we can follow the proof of Lemma~\ref{lem:maincalc} and just ignore the $\beta$ curves (respectively the $\alpha$ curves). Specifically let $f=D_\alpha^{N+1}$. Consider the set of $N+1$ curves in $\Sigma \times \left[0,1\right] \subset N_{\text{id}}$
$$\mathcal{S}_\alpha = \{  \alpha \times \{(2i + 1)/(2N+2)\}~|~ 0\leq i \leq N \}.$$
Let $X$ be the $4$-manifold obtained by attaching $N+1$ $2$-handles to $N_{\text{id}} \times I$ along the curves in $\mathcal{S}_\alpha \times \{1\} \subset N_{\text{id}} \times \{1\}$, each with $+1$ framing.  Then $\partial X = \overline{N}_{\text{id}} \sqcup N_f.$ 
Let $W = X \cup \overline{E}$ where $E$ is the boundary connected sum of $2g$ copies of $S^1 \times B^3$. Just as in the proof of Lemma~\ref{lem:maincalc}, the coefficient system extends to $W$ so
\begin{equation}\rho_\omega(f) = \sigma(W,\mathbb{C}_\Phi) - \sigma_0(W).
\end{equation}
As above $\sigma_0(W)=N+1$. Now we consider $H_2(W;\mathbb{C}_\Phi)$. Since, in the proof of Lemma~\ref{lem:maincalc}, we only slid $\alpha$ curves over other $\alpha$ curves, we see that a matrix for the twisted intersection form on $W$ is given by ignoring, in the matrix of \eqref{int_matrix}, the rows and columns corresponding to the $\beta$ curves. Thus the twisted intersection form on $W$ is given by the matrix $A$ whose signature is just its ordinary signature, which is $N$. Hence
$$
\rho_\omega(D_\alpha^{N+1})=N-(N+1)=-1,
$$
for any $\omega$ of norm $1$ ($\omega\neq 1$). The proof for the $D_\beta$ is the same. This completes the proof of Lemma~\ref{lem:boundedoncyclic}.
\end{proof}

\end{proof}

The proofs above indicate that the same will hold for any subgroup of $\mathcal{K}_g$ containing two Dehn twists on sufficiently different bounding curves.

\section{More on the $\rho$ and $\sigma$-invariants as elements of group cohomology}\label{sec:groupcohoaspects}

The question arises as to whether or not, for a fixed $H\lhd F\equiv \pi_1(\Sigma)$, the higher-order $\rho$-invariants (as $\psi$ varies) lift to classes in $H^1(J(H);\R)$; and whether or not the higher-order signature $2$-cocycles yield non-zero classes in $H^2(J(H);\Z)$. At this time we are only able to comment on these questions in the cases where the unitary representation is finite-dimensional. So, for the remainder of this section we assume that $\psi:F/H\to U(n)$ is a finite-dimensional unitary representation.  In this case $[\sigma_\psi]\in H^2(J(H);\Z)$ by Corollary~\ref{cor:sigcohomologyclass}. The first question we address is: For which $H$ and $\psi$ are these classes non-zero? We abbreviate $J(H)$ by $J$. Note that, in this case, by Proposition~\ref{prop:rhoissigdefect}:

\begin{lemma}\label{lem:rhomodZ} If $\psi$ is a finite-dimensional representation then
the reduction of $\rho_\psi$ mod $\Z$ is a homomorphism $\overline{\rho}_\psi:J\to \R/\Z$ and hence represents a class, denoted $[\overline{\rho}]$ in $H^1(J;\R/\Z)$.
\end{lemma}

Therefore the second question we address is: For which $H$ and $\psi$ are these classes non-zero, and when do they lift to $H^1(J;\R)$? It is enlightening to consider the following subgroup:

\begin{definition}\label{def:B(H)} Let $B(H)\lhd J(H)$ denote the normal subgroup
consisting of those classes $f\in J(H)$ for which the pair $(N_f,\phi_f:\pi_1(N_f)\to F/H)$ is the boundary of some $(W,\tilde{\phi}_f:\pi_1(W)\to F/H)$, where $W$ is a compact oriented $4$-manifold.
\end{definition}

The important observation is that $\rho_\psi$ is integer-valued when restricted to $B(H)$, by Theorem~\ref{thm:APS}. Hence $\overline{\rho}:J\to \R/\Z$ is zero when restricted to $B(H)$ and so $\overline{\rho}$ descends to a well-defined homomorphism on $J/B$ (we abbreviate $B(H)$ by $B$) denoted $\widetilde{\overline{\rho}}$. Consider the following commutative diagram.  The rows are pieces of Bockstein exact sequences.

$$
\begin{diagram}\label{diag:}\dgARROWLENGTH=1.2em
\node{H^1(J/B;\R)}\arrow{e,t}{\tilde{p}}\arrow{s,r}{\pi_1}\node{H^1(J/B;\R/\Z)}\arrow{s,r}{\pi_2}\arrow{e,t}{\tilde{\beta}}\node{H^2(J/B;\Z)}\arrow{s,r}{\pi_3}\arrow{e,t}{\tilde{j}}
\node{H^2(J/B;\R)}
\\
\node{H^1(J;\R)}\arrow{e,t}{p}\node{H^1(J;\R/\Z)}\arrow{e,t}{\beta}\node{H^2(J;\Z)}\arrow{e,t}{j^*}
\node{H^2(J;\R)}
\end{diagram}
$$
We have $[\overline{\rho}]=\pi_2([\widetilde{\overline{\rho}}])$ as observed above. It is not difficult to check that $\beta([\overline{\rho}])=[\delta(\rho)]=[\sigma]$ as expected. Now, using the diagram, we come to our first useful observation.

\begin{lemma}\label{lem:torsioninimage} The torsion classes $[\sigma_\psi]$ lie in the image of the map:
$$
\pi_3:H^2(J/B;\Z)\to H^2(J;\Z).
$$
\end{lemma}

Now let $K(F/H,1)$ denote an Eilenberg-Maclane space of type $(F/H,1)$ and let $\Omega_3(K(F/H,1))$ denote the oriented bordism group  ~\cite[p.216]{DaKi}. Furthermore, observe that there is a well-defined map:
$$
\eta_H: J(H)\rightarrow \Omega_3(K(F/H,1))\cong H_3(F/H;\Z),
$$
given by $\eta_H(f)=(\phi_f)_*([N_f])$, the image of the fundamental class of $N_f$ under the  map induced by $\phi_f$. This was considered by Morita and Heap in the case that $H$ is a term of the lower central series ~\cite{Mor4,Heap}. In particular the proof of Heap's ~\cite[Theorem 4]{Heap} is very general and shows that our $\eta_H$ is a homomorphism. Note that $B(H)$ is (by definition) the kernel of $\eta_H$ so
$$
\eta_H: J/B \hookrightarrow H_3(F/H;\Z)
$$
is a monomorphism.

\begin{proposition}\label{prop:cocycleszero} If $H_3(F/H;\Z)$ is torsion-free (for example if $H$ is a term of the lower central series of $F$ ~\cite[Corollary6.5]{IgOrr}) and $\psi$ is a finite-dimensional representation then the signature cocycles are null-homologous, i.e.  $[\sigma_\psi]=0$.
\end{proposition}

\begin{proof} If $H_3(F/H;\Z)$ is torsion-free then $J/B$ is a torsion-free abelian group. Thus $H^2(J/B;\Z)$ is torsion-free, so $\tilde{j}$ is injective. It follows that $\tilde{\beta}$ is the zero map. Thus
$$
\sigma=\beta\circ \pi_2([\widetilde{\overline{\rho}}])=\pi_3\circ \tilde{\beta}([\widetilde{\overline{\rho}}])=0.
$$
\end{proof}

\begin{proposition}\label{prop:reducedrhofingen} If $H_3(F/H;\Z)$ is  finitely-generated and free abelian (for example if $H$ is a term of the lower central series of $F$) and $\psi$ is a finite-dimensional representation then the classes $[\overline{\rho}]\in H^1(J;\R/\Z)$ lift to $H^1(J;\R)$ and form a (finitely-generated) subgroup of the image of
$$
\pi_1:H^1(J/B;\R)\to H^1(J;\R).
$$
\end{proposition}

\begin{proof} By the proof of Proposition~\ref{prop:cocycleszero}, $\tilde{\beta}=0$ and $\beta([\overline{\rho}])=0$ so any $[\overline{\rho}]$ lifts to $H^1(J;\R)$ and lies in the image of $\pi_1$. If $H_3(F/H;\Z)$ is  finitely-generated then so are $J/B$ and $H^1(J/B;\R)$.
\end{proof}

\begin{remark}\label{rem:1}If $H=[F,F]$ and $J=\mathcal{I}$, then $B=\mathcal{K}$ and
$$
\eta_H:J/B\hookrightarrow H_3(Z^{2g})\cong \bigwedge^3(Z^{2g})
$$
is identifiable with the Johnson homomorphism (see, for example, ~\cite[Theorem 16]{Heap}). It is also known that the map $\pi_1$ above is an isomorphism in this case, since $[\mathcal{I},\mathcal{I}]$ is the radical of $\mathcal{J}$ ~\cite{DJ1}. Hence
$$
H_1(\I;\mathbb{R}/\Z)\subset \bigwedge^3\left((\mathbb{R}/\Z)^{2g}\right)
$$
with known image (corresponding to the known image of $\eta_H$). It would be interesting to know if our $\overline{\rho}_\psi$ in this case span the entire group $H^1(\I;\mathbb{R}/\Z)$.
\end{remark}

\section{Further Methods of Calculation and Relations with Links}\label{sec:calc method}

Suppose $\partial \Sigma$ is connected and $\Sigma '\subset \Sigma$ is a connected compact sub-surface with possibly multiple boundary components.  Then the inclusion $i$ induces a homomorphism $\theta: \mathcal{M}(\Sigma ')\to \mathcal{M}(\Sigma)$, extending by the identity. We assume that one boundary component of $\Sigma '$ intersects $\partial\Sigma$ at the base point. We also assume that, except at the basepoint,  each boundary component of $\Sigma'$ either coincides with a boundary component of $\Sigma$ or is disjoint from $\partial \Sigma$. Suppose $H'$ is a characteristic subgroup of $F'=\pi_1(\Sigma')$ and $H$ is a characteristic subgroup of $F=\pi_1(\Sigma)$ such that $i_*(H')\subset H$.  Fix a unitary representation $\psi: F/H \rightarrow U(\mathcal{H})$ as always. Then there is an induced unitary representation
$$
\psi':\pi_1(\Sigma')/H'\overset{i_*}{\to}F/H\rightarrow U(\mathcal{H}).
$$
If $g\in J(H')$ then one can easily check that $\theta(g)\in J(H)$. Therefore there are induced representations on $\pi_1(N_{\theta(g)})$ and $\pi_1(N_g)$ that factor through $\psi$ and $\psi'$. Hence both  $\rho_\psi(\theta(g))$ and $\rho_{\psi'}(g)$ are defined. The following is then not surprising.

\begin{proposition}\label{prop:subsurfacerestrict} Given $\psi$, $\Sigma'$ and $g$ as above, if $i_*:H_1(\Sigma';\Z)\to H_1(\Sigma;\Z)$ is injective then
$$
\rho_\psi(\theta(g))=\rho_{\psi'}(g).
$$
\end{proposition}
\begin{proof} The proof is very similar to the proof of Theorem~\ref{thm:additivity}. In analogy to the proof of Theorem~\ref{thm:sigbounded}, we will define a certain $4$-manifold $W$ and show that
$$
\partial W= N_{id}\times \{0\}\sqcup N_{g}'\times \{0\}\sqcup -N_{\theta(g)}.
$$
Here we mean $id:\Sigma\to \Sigma$. A superscript prime will denote objects associated to the subsurface $\Sigma'$. Before defining $W$, certain remarks will be helpful.

In this proof it is convenient to take the definition of the mapping torus of $f$ (any $f$) to be the quotient $\Sigma\times [-1,1]/\sim$ where $(x,-1)\sim (f(x),1)$. Recall that we have described the transition from $M_f$ to $N_f$ in terms of Dehn fillings. In this proof it is convenient to consider the alternative definition wherein $N_f$ is obtained as a quotient space of $M_f$ wherein, for each point $x\in \partial \Sigma$, the circle $x\times S^1$ is pinched to  single point. Let $(M_f)_p\equiv N_f$ denote such a pinching operation. Similarly we let $(\Sigma\times [-\epsilon,\epsilon])_p$ denote the quotient space of $\Sigma\times [-\epsilon,\epsilon]$ obtained by pinching each $x\times [-\epsilon,\epsilon]$ to a point (say $x\times \{0\}$). Observe that there is a homeomorphic copy of $(\Sigma\times [-\epsilon,\epsilon])_p$ embedded in $\Sigma\times [-\epsilon,\epsilon]$ obtained using a collar on $\partial \Sigma$. Since there is a copy of $\Sigma\times [-\epsilon,\epsilon]$ embedded in $M_f$ (for any $f$), there is a copy of $(\Sigma\times [-\epsilon,\epsilon])_p$ embedded in $N_f$. 

By the same argument there is a copy of $(\Sigma'\times [-\epsilon,\epsilon])_p$ embedded in $N'_g$.  We claim that there is also a copy of $(\Sigma'\times [-\epsilon,\epsilon])_p$ embedded in $N_{id}$.  Indeed, for \textit{any} $f$, certainly there is a copy of $\Sigma'\times [-\epsilon,\epsilon]$ embedded in $M_{f}$, so there is a copy of 
$\Sigma'\times [-\epsilon,\epsilon]/\sim$ in $N_{f}$, where $\sim$ denotes that we have pinched only those circles corresponding to points $x\in \partial \Sigma'\cap \partial \Sigma$. This is not the same as $(\Sigma'\times [-\epsilon,\epsilon])_p$. However, there is a copy of the latter embedded in $\Sigma'\times [-\epsilon,\epsilon]/\sim$ (and hence in $N_{f}$).

Armed with these observations, we define the cobordism $W$ as the union of $N_{id}\times [0,1]$ and $ N'_g\times [0,1]$ identified along the above copies of $(\Sigma'\times [-\epsilon,\epsilon])_p$:
$$
(\Sigma'\times [-\epsilon,\epsilon])_p\hookrightarrow N_{id}\times \{1\}~~\text{and}~~
(\Sigma'\times [-\epsilon,\epsilon])_p \hookrightarrow  N'_{g}\times \{1\}.
$$
Clearly $\partial W= N_{id}\times \{0\}\sqcup N_{g}'\times \{0\}\sqcup  Y$, and we claim that $Y\cong N_{\theta(g)}$. To see this note that $Y$ is the union of
$$
N_{id}\setminus (\Sigma'\times [-\epsilon,\epsilon])_p\cup N'_g\setminus (\Sigma'\times [-\epsilon,\epsilon])_p,
$$
along their common boundaries. But $N_{\theta(g)}$ has an identical decomposition. For,  one may obtain a copy of $M_{\theta(g)}$ by starting with $M_{id}$, then deleting the product $\Sigma'\times [-\epsilon,\epsilon]$ and replacing it with the ``twisted product'' 
$$
(\Sigma'\times [-\epsilon,0]\cup \Sigma'\times [0,\epsilon])/\sim
$$
where $(x,0)\sim (g(x),0)$. The latter is homeomorphic to the twisted product obtained from  $M'_g$ by deleting a product $\Sigma'\times [-\epsilon,\epsilon]$. After taking into account the relevant pinching, this shows that $Y\cong N_{\theta(g)}$.

The representations on $N_{id}$ and $ N'_g$
extend to $\pi_1(W)$. Hence by Theorem~\ref{thm:APS}
$$
\rho_\psi(id)+\rho_{\psi'}(g)-\rho_\psi(\theta(g))
$$
is the signature defect of $W$. But consider the Mayer-Vietoris sequence as in the proofs of Theorems~\ref{thm:additivity} and ~\ref{thm:sigbounded}:
$$
H_2(N_{id}\times [0,1])\oplus H_2(N'_g\times [0,1])\overset{(i^2_*+j^2_*)}{\longrightarrow} H_2(W)\overset{\partial_*}{\longrightarrow} H_1(\Sigma') \overset{(i^1_*,j^1_*)}{\longrightarrow} H_1(N_{id})\oplus H_1(N'_g).
$$
We claim that $i^1_*$ is injective with any coefficients. Since $\pi_1(N_{id})\cong \pi_1(\Sigma)$, $H_1(N_{id})\cong H_1(\pi_1(\Sigma))$ with any coefficients. Thus it suffices to consider the map on first homology induced by $i:\Sigma'\hookrightarrow \Sigma$. Since $\Sigma'$ and $\Sigma$ deformation retract to $1$-complexes, the hypothesis that this map induces a monomorphism on $H_1(-;\Z)$ is equivalent to saying that, up to homotopy equivalence, $(\Sigma,\Sigma')$ is a $1$-dimensional relative CW-complex. It follows that $H_2(\Sigma;\Sigma')$ is zero with any coefficients and so $i_*$ is injective on $H_1$ with any coefficients.
Hence $H_2(W)$ is supported by $\partial W$ so the twisted and ordinary signatures vanish for $W$. Since, by Corollary~\ref{cor:easyrhoprops}, $\rho_\psi(id)=0$  the desired result follows.
\end{proof}

Proposition~\ref{prop:subsurfacerestrict} and Example~\ref{ex:easiestJ(H)} may be used to calculate certain $\rho$-invariants in terms of well-studied invariants of links of circles in $S^3$. In particular let $\Sigma'=D_n$ be the closed oriented 2-disk with $n$ open subdisks deleted.  Let $\mathcal{M}(D_n)$ denote the group of isotopy classes of orientation-preserving homeomorphisms of $D_n$ that are the identity on $\partial D_n$. It is known that $\mathcal{M}(D_n)$ is isomorphic to the group of n-string \emph{framed}
pure braids, $PF(n) \cong \mathbb{Z}^n \oplus P(n)$ \cite{PS,Natov}. Here $P(n)$ is the usual group of n-string pure braids. Any embedding of $D_n$ into $\Sigma$ defines a homomorphism $\theta: \mathcal{M}(D_n) \rightarrow \mathcal{M}(\Sigma)$. Suppose $i_*:H_1(D_n;\Z)\to H_1(\Sigma;\Z)$ is injective. Then Proposition~\ref{prop:subsurfacerestrict} shows that the $\rho$-invariants associated to $\Sigma$ are equal to  $\rho$-invariants associated to $D_n$, and Example~\ref{ex:easiestJ(H)} indicates how the latter are equal to certain $\rho$-invariants of the zero framed surgery on the link obtained as the closure of the associated pure braid. The latter have been well studied in recent years by knot theorists.

\section{Extension of the $\rho$-invariants to homology cylinders}\label{sec:invariantsforcylinders}

The monoid of homology cylinders may be considered to be an enlargement of the mapping class group of $\Sigma$. In many cases the higher-order $\rho$-invariants and signature co-cycles extend to this monoid. We will focus attention of the case that $\partial \Sigma$ is connected and $H$ is one of the terms of the lower central series of $\pi_1(\Sigma)$.

We recall the definition, following Levine \cite{Le8}.

\begin{definition}\label{def:homcyl}A \textbf{homology cylinder} over $\Sigma$, denoted $C$, is
a compact oriented 3-manifold $C$ equipped
with two embeddings
$i^+ , i^-: \Sigma \rightarrow \partial C$ satisfying that
\begin{enumerate}
\item $i^+$ is
orientation-preserving and $i^-$ is orientation-reversing,
\item $\partial C= i^+ (\Sigma) \cup i^- (\Sigma)$ and
$i^+ (\Sigma) \cap i^- (\Sigma)=i^+ (\partial \Sigma)
=i^- (\partial \Sigma)$,
\item $i^+ \bigl|_{\partial \Sigma}=i^- \bigl|_{\partial \Sigma}$,
\item $i^+,i^-:H_{\ast} \Sigma \rightarrow H_{\ast} C$
are isomorphisms.
\end{enumerate}

\end{definition}

\begin{example}\label{ex:cylinder} For any mapping class $f$,
$(C,i^+,i^-)=(\Sigma \times I, Id \times 1, f \times 0)/\sim$ gives
a homology cylinder, where $\sim$ means that we identify $(x,t)$ to $(x,0)$ for each $t\in [0,1]$ and $x\in \partial\Sigma$.
\end{example}

The set $\mathcal{C}$ of orientation-preserving
diffeomorphism classes of
homology cylinders over $\Sigma$ is a monoid (by concatenation), denoted $\mathcal{C}$, with the identity element
$1_{\mathcal{C}}:=(\Sigma \times I, Id \times 1, Id \times 0)$.
Example \ref{ex:cylinder} shows how to define a map $\mathcal{I}\to \mathcal{C}$ that is an injective map of monoids.

For any $C\in\mathcal{C}$ then there is an associated closed \emph{oriented} manifold $N_C$ obtained by identifying the two copies of $\Sigma$. If $\overline{C}$ is the homology cylinder obtained by reversing the roles of $+$ and $-$ then $N_{\overline{C}}=-N_C$. If $C$ lies in the image of $f\in\mathcal{I}$ as in Example~\ref{ex:cylinder} then $N_C\cong N_f$. Given $H\lhd \pi_{1}(\Sigma)$, we say that \textbf{$C$ induces the identity modulo $H$} if, for all $x\in \pi_1(\Sigma)$, $i^+_*(x)=i^-_*(xh)$  for some $h\in H$. We then say \textbf{$C\in C(H)$}. Thus, for example, $\mathcal{C}(F_{2})$ is the analogue of the Torelli group. Then we have
$$
\pi_{1}(N_{C})=\pi_{1}(C)/\langle i^{+}_{\ast}(x)=i^{-}_{\ast}(x)\mbox{ for all }x\in\pi_{1}(\Sigma)\rangle
$$
For example, if $H=F_{2}$ and $C\in C(H)$, then $H_{1}(N_{C})\cong \Z^{2g}$ coming from $H_{1}(\Sigma)$.

Consider the case $H=F_{n}$, where $F=\pi_{1}(\Sigma)$ and assume $C\in C(F_{n})$.  By Stallings' Theorem \cite[Theorem 5.1]{St}, $i^{\pm}$ induce isomorphisms
$$
F/F_{n}\xrightarrow{i^{+}_{n}} \pi_{1}(C)/(\pi_{1}(C))_{n}\xleftarrow{i^{-}_{n}} F/F_{n}.
$$
Moreover, since $C\in C(H)$, $i^{+}_{n}\circ(i^{-}_{n})^{-1}$ is the identity on $F/F_{n}$.  Then we have
\begin{eqnarray*}
\pi_{1}(N_{C})/(\pi_{1}(N_{C}))_{n} &\cong& \pi_{1}(C)/\langle i^{+}_{\ast}(x)=i^{-}_{\ast}(x), \forall x\in F, (\pi_{1}(C))_{n}\rangle\\
&\cong& \pi_{1}(C) / \langle i^{-}_{\ast}(x)i^{-}_{\ast}(h_x)=i^{-}_{\ast}(x), \forall x\in F, (\pi_{1}(C))_{n}\rangle\\
&\cong& \pi_{1}(C) / \langle i^{-}_{\ast}(h)=1, h_x\in F_{n}, (\pi_{1}(C))_{n}\rangle\\
&\cong& \pi_{1}(C)/ (\pi_{1}(C))_{n}
\end{eqnarray*}

Thus, for $C\in C(F_{n})$, there is a unique epimorphism
$$
\phi_{C}: \pi_1(N_{C})\to F/F_{n}
$$
 that is the composition of
\begin{equation}\label{eq:defrhon}
\pi_{1}(N_{C})\twoheadrightarrow \pi_{1}(N_{C)}/(\pi_{1}(N_{C}))_{n}\xrightarrow{\cong}\pi_{1}(C)/(\pi_{1}(C))_{n}\xrightarrow{(i^{+}_{n})^{-1}} F/F_{n}
\end{equation}
Therefore, given a fixed unitary representation $\psi: F/F_{n}\to U$, we can define $\rho^{\psi}(C)=\rho(N_{C},\psi\circ\phi_{C})$.  In the infinite-dimensional case, we will denote this invariant $\rho_{n}(C)$ (using the left-regular representation of $F/F_{n}$). Moreover, the restriction to $\mathcal{C}(F_n)$ is not necessary, since we can extend $\rho_n$ to all of $\mathcal{C}$ by
\begin{definition}\label{def:rhon}  If $C\in \mathcal{C}$ then $\rho_n(C)$ is $\rho(N_C,\psi_C)$ where $\psi_C$ is the composition
$$
\pi_{1}(N_{C})\twoheadrightarrow \pi_{1}(N_{C)}/(\pi_{1}(N_{C}))_{n}^r\overset{\ell_r}{\longrightarrow}U\left(\ell^{(2)}(\pi_{1}(N_{C)}/(\pi_{1}(N_{C}))_{n}^r\right),
$$
and $G^r_n$ denotes the $n^{th}$ term of the \emph{rational lower central series} \cite{St}.
\end{definition}

We also consider a quotient of $\mathcal{C}$, \textbf{the group, $\mathcal{H}$, of homology cobordism classes of homology cylinders }, wherein $C$ is homology cobordant to $D$ if there is a compact oriented $4$-manifold $V$ whose boundary is $N_{\overline{C}\circ D}$  such that the natural inclusions $C\hookrightarrow V$ and $D\hookrightarrow V$ induce isomorphisms on homology (for the details of this definition we refer the reader to ~\cite{Le8,Le9}). The composition $\mathcal{I} \to \mathcal{C}\to \mathcal{H}$ is a monomorphism of groups. We will denote the group of homology cobordism classes of homology cylinders that induce the identity modulo $F_{n}$ by $\mathcal{H}(F_{n})$.

We will now show that the $\rho_n$ of Definition~\ref{def:rhon} are homology cobordism invariants and hence descend to $\mathcal{H}$ (again by Stallings theorem ~\cite{St}).

\begin{theorem}\label{thm:rhondefoncyl}
The invariant $\rho_{n}:C(F_2)\to \mathbb{R}$ descends to a well-defined function $$\rho_{n}:\mathcal{H}(F_{2})\to\mathbb{R}$$
\end{theorem}

\begin{proof}
Let $C$ and $D$ be homology cylinders that induce the identity modulo $F_{2}$ and assume $C$ and $D$ are homology cobordant. The first step in the proof is to deduce that the closed manifolds $N_{C}$ and $N_{D}$ are homology cobordant.

Since $C$ and $D$ are homology cobordant, there is a 4-manifold $V$ with $\partial V= N_{\overline{C}\circ D}$ so that the inclusions of $C$ and $D$ into $V$ induce isomorphisms on all homology groups.

Let $W$ denote the 4-manifold obtained by identifying $N_{\overline{C}}\times [0,1]$ and $N_{D}\times[0,1]$ along a product neighborhood of $\Sigma$ in $N_{\overline{C}}\times\{1\}$ and $N_{D}\times\{1\}$.  The boundary of $W$ decomposes as $\partial W= N_{\overline{C}}\sqcup N_{D} \sqcup - N_{\overline{C}\circ D}$. Now let $\displaystyle E= W\bigcup_{-N_{\overline{C}\circ D}} -V$ and observe $\partial E= -N_{C}\sqcup N_{D}$.  We claim that $E$ is the desired homology cobordism between $N_{C}$ and $N_{D}$. It suffices to show that $H_*(E,N_C)=0$ since then, by symmetry, $H_*(E,N_D)=0$.
Clearly $H_{0}(E,N_{C})=H_{4}(E, N_{C})=0$.   By assumption the inclusion-induced map ${i^{+}}_{\ast}:H_{1}(\Sigma)\to H_{1}(N_{C})$ is an isomorphism.  Moreover each of the inclusions, namely of $\Sigma$ into $C $, $D$, $N_C$, $N_D$, $N_{\overline{C}\circ D}$, respectively and subsequently into $V$, $W$ and $E$, induces an isomorphism on $H_1$. Hence $H_{1}(E, N_{C})=H_{1}(E,N_{D})=0$. By duality and the universal coefficient theorem, we have
$$
H_{3}(E,N_{C})\cong H^{1}(E, N_{D})\cong \mathrm{Hom}(H_{1}(E,N_{D}),\Z)=0.
$$
Similarly, $H_2(E,N_C)$ is torsion-free. Thus to show that $H_2(E,N_C)=0$ it now suffices to show that $\chi(E,N_{C})=0$.  

By the long exact sequence for the pair $(E,N_{C})$,
$$
\chi(E,N_{C})=\chi(E)-\chi(N_{C})=\chi(E)=\chi(W)+\chi(V)-\chi(N_{\overline{C}\circ D})=\chi(W)+\chi(V),
$$
since the Euler characteristic of a closed oriented $3$-manifold is zero.  But $W$ is homotopy equivalent to $N_{C}\cup_{\Sigma}N_{D}$, hence
$$
\chi(W)=\chi(N_C)+\chi(N_D)-\chi(\Sigma)=2g-1;
$$
and , since $H_*(\Sigma)\cong H_*(C)\cong H_*(V)$,  ~$\chi(V)=1-2g$. Thus $\chi(E,N_C)=0$.

This completes the first step of the proof,  namely that $E$ is a homology cobordism between $N_C$ and $N_D$.

The second step of the proof is to show that the $\rho_n$ are, loosely speaking, invariants of homology cobordism of $3$-manifolds.  Suppose  that $N_{C}$ and $N_{D}$ are homology cobordant via the 4-manifold $E$ from above.  Let $\Gamma=\pi_{1}(N_{C})$, $\Delta=\pi_{1}(N_{D})$, $G=\pi_{1}(E)$, and $\gamma: N_{C}\to E$ and $\delta:N_{D}\to E$ denote the inclusion maps.  We have the following commutative diagram, where the maps on the bottom row are isomorphisms by Stallings' Theorem \cite[Theorem 7.3]{St}:
 $$
\begin{diagram}\label{diag:commutes}\dgARROWLENGTH=1.2em
\node{\pi_1(N_C)}\arrow{e,t}{i_*}\arrow{s,r}{\pi}\node{\pi_1(E)}\arrow{s,r}{\pi}\node{\pi_1(N_D)}\arrow{s,r}{\pi}\arrow{w,t}{i_*}
\\
\node{\frac{\pi_1(N_C)}{\pi_1(N_C)_n^r}}\arrow{se,b}{\ell_r}\arrow{e,tb}{j^n_*}{\cong}
\node{\frac{\pi_1(E)}{\pi_1(E)_n^r}}\arrow{s,r}{\ell_r}\node{\frac{\pi_1(N_D)}{\pi_1(N_D)_n^r}}\arrow{sw,b}{\ell_r}\arrow{w,tb}{i^n_*}{\cong} \\ \node[2]{U(\mathcal{H})}
 \end{diagram}
$$

Therefore, by Theorem~\ref{thm:APS},
$$
\rho_n(D)-\rho_n(C)=\sigma^{(2)}(E,\psi)-\sigma(E).
$$
Since $H_*(E,N_C;\Z)=0$,
$$
H_2(E;\Z)\to H_2(E,\partial E;\Z)
$$
is the zero map so $\sigma(E)=0$. Additionally, letting $\G=\pi_1(E)/\pi_1(E)_n^r$, since $H_2(E,N_C;\Z)=0$  and $\G$ is a poly-(torsion-free-abelian group), it follows from ~\cite[Corollary 2.8]{CH2} that $H_2(E,N_C;\Z[\G])$ is a $\Z[\G]$-torsion module, implying that $H_2(E,N_C;\mathcal{K}\G)=0$. Thus
$$
H_2(E;\mathcal{K}\G)\overset{\partial_*}{\rightarrow} H_2(E,\partial E;\mathcal{K}\G)
$$
is the zero map (see the Appendix for definitions and more detail for $\mathcal{K}\G$). Hence
$$
H_2(\partial E;\mathcal{K}\G)\to H_2(E;\mathcal{K}\G)
$$
is surjective. By property $1.$ of Proposition~\ref{prop:rhoprops},
 $\sigma^{(2)}(E,\psi)=0$.  Thus $\rho_{n}(C)=\rho_{n}(D)$.
\end{proof}

The discussion of Section~\ref{sec:Defsigs} extends to homology cylinders so that we can define signature cocycles for homology cylinders. Namely, given $C$ and $D\in \mathcal{C}(F_n)$ we can form a $4$-manifold $W(C,D)$ (analogous to $W(f,g)$) defined as
$$
W(C,D)=N_C\times[0,1]\cup_{\overline{A}\times\Sigma} N_D\times[0,1]
$$
where $\overline{A}$ is the arc $A$ with added collars on its boundary. Then
$$
\partial W(C,D)= N_C\sqcup N_D\sqcup -N_{CD}.
$$
Moreover, the fundamental group of a homology cylinder is a product modulo any term of the lower central series. With this in mind we can define a signature $2$-cocycle on $\mathcal{H}(F_n)$ that extends that which we already defined on $J(F_n)$ in the second part of Definition~\ref{defcocycle}.

\begin{definition}\label{defcocyclecyl} Given $\Sigma$ and $n$, we define a function $\sigma_n^{(2)}:\mathcal{H}(F_n)\times \mathcal{H}(F_n)\to \mathbb{R}$ by
$$
\sigma_n^{(2)}(C,D)= \sigma^{(2)}\left(W(C,D),\widetilde{\psi_n}\right)-\sigma(W(C,D)).
$$
\end{definition}

Then it follows immediately from Theorem~\ref{thm:APS} that

\begin{proposition}\label{prop:rhoissigdefectcyl} For each $n$ and $C,D\in \mathcal{H}(F_n)$,
$$
\sigma^{(2)}_{n}(C, D)=\rho_{n}(C) +  \rho_{n}(D)-\rho_{n}(CD).
$$
where $\rho_n$ is as in Definition~\ref{def:rhon}.
\end{proposition}

Our main result, Theorem~\ref{thm:sigbounded}, continues to hold and so
\begin{corollary}\label{cor:boundedcocyclecly} For any $n$, $\sigma^{(2)}_n$ is a bounded $2$-cocycle on $\mathcal{H}(F_n)$.
\end{corollary}
\begin{proposition} \label{prop:quasihomocyl} For any $n\geq 2$,  $\rho_n$ is a real-valued quasimorphism on $\mathcal{C}(F_n)$ and $\mathcal{H}(F_n)$.
\end{proposition}

Note that one can define  quasimorphisms and  cocyles on the monoid $\mathcal{C}(F_n)$.

We claim that these invariants are quite rich, as indicated by the following theorems. We should clarify that, while $\rho_n$ can be defined on all of $\mathcal{H}(F_2)$, it is only a quasimorphism when restricted to $\mathcal{H}(F_m)$ for $m\geq n$.

\begin{theorem}\label{thm:rhondense} Suppose $\Sigma$ has genus $g\geq 1$ and non-empty boundary. Then, for any $n\geq 2$
\begin{itemize}
\item [1.] The image of $\rho_n:\mathcal{H}(F_n)\to \mathbb{R}$ is dense.
\item [2.] The image of $\rho_n:\mathcal{H}(F_n)\to \mathbb{R}$ is infinitely generated.
\end{itemize}
\end{theorem}
\begin{theorem} \label{thm:indep} Suppose $\Sigma$ has genus $g\geq 1$ and non-empty boundary. Then, for any $m\geq 2$, $\{\rho_n \}_{n=2}^{\infty}$
is a linearly independent subset of the real vector space of all functions $\{f:\mathcal{H}(F_m)\to\mathbb{R}\}$ modulo the subspace of bounded functions.
\end{theorem}
We have learned that, for the case $n=2$, a result identical to Theorem~\ref{thm:rhondense} appeared in the (unpublished) thesis of T. Sakasai. These results parallel \cite[Section 5]{Ha2} where essentially the same results were proved for von Neumann $\rho$-invariants associated to the torsion-free derived series, rather than the lower central series. Before proving these theorems, we need to introduce a technique for modifying a homology cylinder in such a way that the value of $\rho_n$ changes in a predictable manner.

\begin{subsection}{Altering homology cylinders by infection}

Suppose $C$ is a homology cylinder, $\eta$ is a  null-homologous oriented simple closed curve in the interior of $C$, and $K$ is an oriented knot in $S^3$. We describe a procedure for altering $C$ to a new homology cylinder, $C(\eta,K)$, called \textbf{infecting $C$ along $\eta$ using $K$}   \cite[p.406]{Ha2}\cite[Section 3]{COT2}. Let $N(\eta)$ and $N(K)$ denote tubular neighborhoods of $\eta$ in $C$ and $K$ in $S^3$ respectively, and let $\mu_K$, $\ell_K$, $\mu_\eta$, $\ell_\eta$ denote the meridians and longitudes of $K$ and $\eta$. Define
\begin{equation}\label{eq:infection}
C(\eta,K)=(C-N(\eta))\cup_f (S^3-N(K))
\end{equation}
where $f:\partial(S^3-N(K))\to \partial(C-N(\eta)$ is defined by $f(\mu_K)=\ell_\eta^{-1}$ and $f(\ell_K)=\mu_\eta$. Since we have formed $C(\eta,K)$ by excising $N(\eta)$ and replacing it with $S^3-N(K)$, both of which have the homology of a circle, $C(\eta,K)$ remains a homology cylinder. Indeed, we may think of the solid torus $N(\eta)$ as the exterior of the trivial knot, $U$, in $S^3$. Then, since there is a degree one map relative boundary from $S^3-K$ to $S^3-U$, there is a degree one map relative boundary $C(\eta,K)\to C$. We leave it to the reader to check that if $C\in \mathcal{C}(F_n)$ then $C(\eta,K)\in \mathcal{C}(F_n)$.

The process of infecting a homology cylinder using a knot $K$ alters its $\rho$-invariants by an additive factor equal to the average of the classical Levine-Tristram signatures of $K$. Recall that if $K\hookrightarrow S^3$ and $V$ is a Seifert matrix for $K$ then, for any complex number $\omega$ of norm $1$, $(1-\omega)V+(1-\overline{\omega})V^T$ is a hermitian matrix whose signature is called the Levine-Tristram $\omega$-signature of $K$. The average of these integers, which is the integral over the circle, is denoted $\rho_0(K)\in \mathbb{R}$. The following proof closely follows \cite[Theorem 5.8]{Ha2} where the same theorem is proved for von Neumann $\rho$-invariants associated to the torsion-free derived series.

\begin{proposition}\label{prop:diffrho}Let $C(\eta,K)$ be as defined above
and let $G=\pi_1(N_C)$.  If, for some
$n\geq 1$, $\eta \in G_{n-1}$ but no power of $\eta$ lies in $G_{n}$, then
$$\rho_i(C(\eta,K))-\rho_i(C)=
\left\{
\begin{array}{ll}
    0 & 2 \leq i\leq  n-1;\\
    \rho_0(K) & i \geq n.\\
\end{array}
\right.
$$
where $\rho_0(K)$ is the integral of the classical Levine-Tristram signature function of $K$.
\end{proposition}

\begin{proof}[Proof of Proposition~\ref{prop:diffrho}] We construct a cobordism, $W$, relating $N_{C(\eta,K)}$ to $N_C$ as follows. Let $M_K$ denote the zero framed Dehn surgery on $S^3$ along the knot $K$. Recall that this is defined as
$$
M_K=S^3-N(K)\cup_g (S^1\times D^2)
$$
where $g$ is an orientation-reversing diffeomorphism of the torus that identifies $\{1\}\times \partial D^2$ with $\ell_K$. The we define
\begin{equation}\label{eq:Wdef}
W=\left(N_C \times [0,1]\right)\cup_h M_K\times [0,1],
\end{equation}
where $h$ identifies the solid torus $N(\eta)\times \{1\}$ with the solid torus $S^1\times D^2\times \{0\}\subset M_K\times \{0\}$, as indicated schematically in Figure~\ref{fig:3boundcomps} ($N(\eta)\times \{1\}$ is dashed).

\begin{figure}[htbp]
\begin{center}
\begin{picture}(165,100)
\put(0,0){\includegraphics{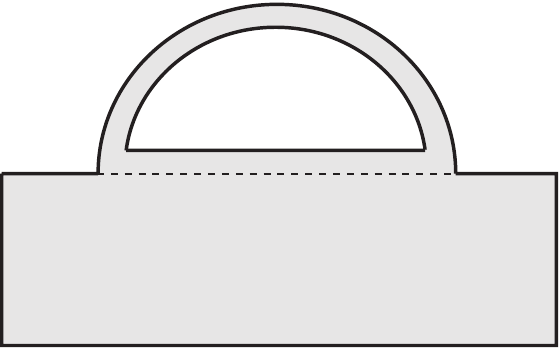}}
\put(65,-12){$N_C \times \{0\}$}
\put(-15,63){$N_{C(\eta,K)}$}
\put(63,37){$N(\eta)\times \{1\}$}
\put(63,63){$M_K\times \{1\}$}
\end{picture}
\end{center}
\caption{The $4$-manifold $W$ with $\partial W = N_C \sqcup -N_{C(\eta,K)} \sqcup M_K$}\label{fig:3boundcomps}
\end{figure}
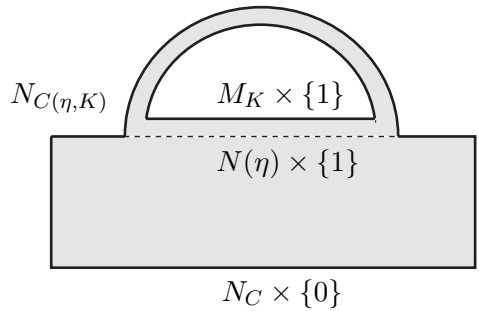
It follows that
$$
\partial W =
N_C \sqcup -N_{C(\eta,K)} \sqcup M_K.
$$
Let $E=\pi_1(W)$, and $\G_i=E/E_{i}$ and consider the coefficient system
$$
\psi:E\overset{\phi}{\to}\G_{i}\overset{\ell_r}{\to}U(\ell^{(2)}(\G)_i)
$$
where $\phi$ is the canonical projection and $\ell_r$ is the left-regular representation. Then, by Theorem~\ref{thm:APS},
\begin{equation}\label{eq:rhosig1}\rho(N_C,\psi)-\rho(N_{C(\eta,K)},\psi) +
\rho(M_K,\psi) = \sigma^{(2)}(W,\psi)-\sigma(W).
\end{equation}
We claim that the right-hand side of \eqref{eq:rhosig1} is zero. In fact this is a direct consequence of \cite[Lemma 2.4]{CHL3}, so we will not repeat the proof. The basic idea is to show, using the Mayer-Vietoris sequence with $\mathcal{K}\G_i$-coefficients associated to \eqref{eq:Wdef}, that  $H_2(\partial W;\mathcal{K}\Gamma_i) \rightarrow
H_2(W;\mathcal{K}\Gamma_i)$ is surjective and then apply property $1.$ of Proposition~\ref{prop:rhoprops}.

Let $P=\pi_1(N_{C(\eta,K)})$ and recall $G=\pi_1(N_C)$. We claim that the inclusion maps $N_{C(\eta,K)}\hookrightarrow W$ and $N_{C}\times\{0\}\hookrightarrow W$ induce isomorphisms
\begin{equation}\label{eq:pi1}
P/P_i\cong E/E_i=\G_i~~ \text{and}~~ G/G_i\cong E/E_i=\G_i
\end{equation}
for each $i$. To see the first, note that $W$ deformation retracts to $\overline{W}=N_C\times\{0\}\cup N_{C(\eta,K)}$. Moreover $\overline{W}=N_{C(\eta,K)}\cup N(\eta)\times\{1\}$. Therefore $\overline{W}$ can be obtained from $N_{C(\eta,K)}$ by adding a single $2$-cell and then a single $3$-cell. The $2$-cell is added along $\ell_K$. But recall that, for a knot exterior, the lower central series stabilizes at the commutator subgroup. Thus $\ell_K\in \pi_1(S^3-N(K))_i$ for all $i$ and so $\ell_K\in P_i$ for all $i$. This implies the first isomorphism of \eqref{eq:pi1}. For the second inclusion, note that by the Seifert-Van Kampen theorem,
$$
E=\pi_1(W)\cong G\ast_\Z\pi_1(M_K),
$$
where $\eta$ is identified with $\mu_K$. The abelianization map $\pi_1(M_K)\to \Z$ induces a retraction $r$
$$
G\to E\cong G\ast_\Z\pi_1(M_K)\overset{r}{\to} G\ast_\Z\Z\cong G
$$
whose kernel is the normal closure of the commutator subgroup $\pi_1(S^3-K)_2\cong \pi_1(S^3-K)_i$. Thus $E/E_i\cong G/G_i$  establishing the second isomorphism of \eqref{eq:pi1}.

Therefore, by \eqref{eq:defrhon} and property $2.$ of Proposition~\ref{prop:rhoprops},
$$
\rho(N_C,\psi)=\rho_i(C)~~ \text{and}~~ \rho(N_{C(\eta,K)},\psi)=\rho_i(C(\eta,K)).
$$

Hence \eqref{eq:rhosig1} becomes
\begin{equation}\label{eq:rhosig2}
\rho_i(C(\eta,K))-\rho_i(C)
= \rho(M_K,\psi).
\end{equation}

It remains only to analyze $\rho(M_K,\psi)$. Recall that $\pi_1(M_K)$ is normally generated by the meridian $\mu_K$, which is identified with $\eta$ under the infection process. Since, by hypothesis,
$\eta\in G_{n-1}$, $\mu_K\in E_{n-1}$ and so $\psi(\pi_1(M_K))=0$ if $i\leq n-1$. Thus, by property $3.$ of Proposition~\ref{prop:rhoprops}, $\rho(M_K,\psi)=0$. Thus \eqref{eq:rhosig2} establishes Proposition~\ref{prop:diffrho} in the case $i\leq n-1$.

Now suppose $i\geq n$ and $i\geq 2$. Since $\pi_1(S^3-K)_2\cong \pi_1(S^3-K)_i$, we have $\psi(\pi_1(S^3-K)_2)=0$. Thus the restriction of $\psi$ to $\pi_1(M_K)$ factors through its abelianization, $\Z=\langle \mu_k\rangle$. Hence it suffices to show that $\psi(\mu_K)=\psi(\eta)$ is of infinite order in $\G_i$. Since $i\geq n$, there is a surjection $\G_i\to \G_{n}=E/E_{n}\cong G/G_{n}$ (using $\eqref{eq:pi1}$). So it suffices to show that no proper power of $\eta$ lies in $G_{n}$. But this was our hypothesis.
Therefore, by property $4.$ of Proposition~\ref{prop:rhoprops}, $\rho(M_K,\psi)=\rho_0(M_K)$, the integral over the circle of the Levine-Tristram signatures of $K$.

This completes the proof of Proposition~\ref{prop:diffrho}.

\end{proof}

\end{subsection}

Now that we can create homology cylinders with varied $\rho_n$, we can easily prove Theorems~\ref{thm:rhondense} and ~\ref{thm:indep}.

\begin{proof}[Proof of Theorem~\ref{thm:rhondense}] For fixed $n\geq 2$, let $C\in \mathcal{C}(F_n)$ be the \emph{identity} homology cylinder. Then
$$
\pi_1(C)/\pi_1(C)_{i}\cong\pi_1(N_C)/\pi_1(N_C)_{i}\cong F/F_i
$$
for every $i$ where $F=\pi_1(\Sigma)$ is a non-abelian free group. Since $F_{n-1}/F_{n}$ is known to be a non-trivial free abelian group, there exists some null-homologous simple closed curve $\eta\in C$ which lies in $\pi_1(N_C)_{n-1}$ but no power of which lies in $\pi_1(N_C)_{n}$. Therefore, by Proposition~\ref{prop:diffrho}, for \emph{any} knot $K$,
$$
\rho_n(C(\eta,K))=\rho_0(K).
$$
hence it suffices to show that
$$
\{\rho_0(K)~~|~~K\hookrightarrow S^3\}
$$
is dense in $\mathbb{R}$ and is an infinitely generated group. Both of these were shown explicitly in \cite[Thm. 5.11]{Ha2} using \cite[Section 2]{ChL}\cite[Prop.2.6]{COT2}.
\end{proof}

\begin{proof}[Proof of Theorem~\ref{thm:indep}] Suppose that $r_1 \rho_{i_1} + \cdots + r_k \rho_{i_k}$ is a function bounded by $D>0$,
where $r_i$ are non-zero real numbers and the $i_{j}$ are increasing with $j$. We shall reach a contradiction. Let $C\in \mathcal{C}(F_m)$ be the \emph{identity} homology cylinder and let $F=\pi_1(N_C)$. Let $n=i_k\geq 2$. As in the proof of Theorem~\ref{thm:rhondense} above, there is a curve $\eta\in C$ such that $\eta\in
F_{n-1}$ but no power of which lies in $F_n$. Consider $C(\eta,K)$ for any $K$ with
$|\rho_0(K)|>D$ (for example, let $K$ be the connected sum of a large number of right-handed
trefoil knots). For any $i\leq n-1$, $\eta\in \pi_1(N_C)_i$ so, by Proposition~\ref{prop:diffrho}, $\rho_i(C(\eta,K))=0$  and $|\rho_n(C(\eta,K))|>D$. This is a contradiction.
\end{proof}

In \cite{Sak1}, Sakasai defined an exact sequence analogous to our \eqref{eq:nonst2}:
\begin{eqnarray}\label{eq:nonstSakasai}
1\to \mathcal{S}_n\overset{i}\longrightarrow \mathcal{H}(F_n)\overset{r_n}\longrightarrow \mathrm{Isom}\left(H_1(\Sigma;\mathbb{Z}[F/F_n])\right),
\end{eqnarray}

It follows from Theorem~\ref{thm:additivity} that

\begin{proposition}
The restriction of $\rho_{n}:\mathcal{H}(F_{2})\to\mathbb{R}$ to $\mathcal{S}_n$ is a homomorphism.
\end{proposition}

\section{Appendix: Definition and basic properties of the von Neuman signature and von Neumann $\rho$-invariants}\label{sec:sigprops}

Given a closed, oriented 3-manifold $M$, a discrete group $\G$, and a representation $\phi : \pi_1(M)
\to \G$, the \textbf{von Neumann
$\boldsymbol{\rho}$-invariant}, $\rho(M,\phi)\in \mathbb{R}$, was defined by Cheeger and Gromov  ~\cite{ChGr1}. It is defined by choosing a Riemannian metric on $M$ and taking the difference between the $\eta$-invariant of $M$ and the von Neumann $\eta$ invariant of the $\G$-covering space associated to $\phi$. However, we prefer an equivalent definition of $\rho$, as a signature defect. Suppose $(M,\phi) = \partial(W,\psi)$ for some compact, oriented 4-manifold $W$ and $\psi :\pi_1(W) \to \G$ (meaning $\phi=\psi\circ i_*$), then it is known that $\rho(M,\phi) =
\sigma^{(2)}_\G(W,\psi) - \sigma(W)$ where $\sigma^{(2)}_\G(W,\psi)$ is the $L^{(2)}$-signature (von Neumann signature) of the $\G$-covering space of $W$ associated to $\psi$. We recall below the definition of the $L^{2}$-signature of a $4$-dimensional manifold.  For more information on $L^{2}$-signature and $\rho$-invariants see \cite[Section 2]{CT}, \cite[Section 5]{COT}\cite{LS}\cite[Section 3]{Ha2}.

Let $\G$ be a countable discrete group. Let $\mathcal{N}\G$ be the group von Neumann algebra of $\G$, a subalgebra of the bounded linear operators on $\ell^{(2)}(\G)$, and let $\mathcal{U}\G$ be the algebra of unbounded operators affiliated to $\mathcal{N}\G$ ~\cite{Rei1}. Let $h_{W,\G}$ be the equivariant intersection form on $H_2(W)$ with $\mathcal{U}\G$-coefficients, which is defined as the
composition
\begin{equation}\label{eq:herm}H_{2}(W;\mathcal{U}\G) \rightarrow H_{2}(W,\partial
W;\mathcal{U}\G) \xrightarrow{\text{PD}}
\overline{H^{2}(W;\mathcal{U}\G)} \xrightarrow{\kappa}
\overline{H_{2}(W;\mathcal{U}\G)^\ast}
\end{equation}
where $H_{2}(W;\mathcal{U}\G)^\ast =
\text{Hom}_{\mathcal{U}\G}(H_{2}(W;\mathcal{U}\G),\mathcal{U}\G)$.
Since $\mathcal{U}\G$ is a von Neumann regular ring (and is endowed with an involution),
the modules $H_{2}(W;\mathcal{U}\G)$ are finitely generated
projective right $\mathcal{U}\G$-modules. $\mathcal{U}\G$ is endowed with an involution with respect to which $h_{W,\G} \in \text{Herm}_{n}(\mathcal{U}\G)$ (a Hermitian matrix). Then $\sigma^{(2)}_{\G}: \text{Herm}_{n}(\mathcal{U}\G) \rightarrow \mathbb{R}$ is defined by
$$\sigma^{(2)}_{\G}(h)=\text{tr}_{
\G}(p_+(h))-\text{tr}_{
\G}(p_-(h))$$
for any  $h \in \text{Herm}_{n}(\mathcal{U}\G)$ where
$\text{tr}_{\G}$ is the von Neumann trace and $p_\pm$ are the
characteristic functions on the positive and negative reals. Here the trace is equal to the $\mathcal{U}\G$-dimension ~\cite{Rei1}. Thus we define $\sigma^{(2)}(W,\G)=\sigma^{(2)}_{\G}(h_{W,\G})$. It is known that $\sigma^{(2)}_{\G}$ descends to the Witt group of Hermitian forms on finitely generated projective
$\mathcal{U}\G$-modules (see for example Corollary~5.7 of \cite{COT}).

Suppose that $\G$ is a poly-(torsion-free-abelian) group. In particular $\G$ is torsion-free and amenable. In this case the von Neumann signature can be defined without the use of $\mathcal{U}\G$. For it is then known that $\Z\G$ is an Ore domain and embeds in its classical right ring of quotients $\mathcal{K}\G$, which is a division ring.  Moreover, the
map from $\Z\G$ to $\mathcal{U}\G$ factors as $\Z\G \rightarrow \mathcal{K}\G \rightarrow \mathcal{U}\G$
making $\mathcal{U}\G$ into a $\mathcal{K}\G-\mathcal{U}\G$-bi-module.  Since any module over a skew field is free, $\mathcal{U}\G$ is a flat $\mathcal{K}\G$-module.
Hence, $H_{2}(W;\mathcal{U}\G)\cong H_{2}(W;\mathcal{K}\G)\otimes_{\mathcal{K}\G}\mathcal{U}\G$. In particular, $H_{2}(W;\mathcal{K}\G)=0$ if and only if
$H_{2}(W;\mathcal{U}\G)$=0. In this case $\sigma^{(2)}_\Gamma$ can be thought of as a homomorphism from
$L^0(\mathcal{K}(\Gamma))$ to $\mathbb{R}$. Aside from the definition, the properties that we use in this paper are:

\begin{proposition}\label{prop:rhoprops}\

\begin{itemize}
\item [1.] If $(M,\phi) = \partial (W,\psi)$ for some compact,
4-manifold $W$ and
$$
H_2(W;\mathcal{U}\Gamma)/\text{Image}(H_2(\partial W;\mathcal{U}\Gamma))
$$
is a finitely-generated free $\mathcal{U}\Gamma$-module containing a free summand of half rank on which the equivariant intersection form vanishes, then
$\sigma^{(2)}_\Gamma(W,\psi)$ vanishes. If $\G$ is poly-torsion-free abelian then the same holds with $\mathcal{K}\G$-coefficients.
\item [2.] If $\phi$ factors through $\phi': \pi_1(M)\to \G'$ where
$\G'$ is a subgroup of $\G$, then $\rho(M,\phi') = \rho(M,\phi)$.
\item [3.] If $\phi$ is trivial (the zero map), then $\rho(M,\phi) = 0$.
\item [4.] If $M=M_K$ is the zero-surgery on a knot $K$ and $\phi:\pi_1(M)\to \mathbb{Z}$ is the abelianization, then $\rho(M,\phi)$ is denoted $\boldsymbol{\rho_0(K)}$ and is equal to the integral over the circle of the Levine-Tristram signature function of $K$ ~\cite[Prop. 5.1]{COT2}. Thus $\rho_0(K)$ is the average of the classical signatures of $K$.
\item [5.] The von Neumann signature satisfies Novikov  additivity, i.e. if $W_1$ and $W_2$ intersect along a common boundary component then $\sigma^{(2)}_\G(W_1\cup W_2)=\sigma^{(2)}_\G(W_1)+\sigma^{(2)}_\G(W_2)$ ~\cite[Lemma 5.9]{COT}.
    \end{itemize}
\end{proposition}

\section{Funding}
This work was supported by the National Science Foundation [DMS-0406050 and DMS-0706929 to T.C., DMS-0539044 and CAREER-DMS-0748458 to S.H., and DMS-0706929 and Mathematical Sciences Postdoctoral Research Fellowship DMS-0902786 to P.H.]; and the Alfred P. Sloan Foundation [to S.H.].

\bibliographystyle{plain}
\bibliography{mybib7mathscinetrevised}
\end{document}